\documentclass[11pt,a4paper,reqno]{amsart}
\usepackage{hyperref}
\usepackage{amssymb, amsmath, amsthm}
\usepackage{bbm}

\numberwithin{equation}{section}
\usepackage{color}
\definecolor{gray}{gray}{0.5}

\def\R{\mathbb{R}}
\def\N{\mathbb{N}}

\def\Z{\mathbb{Z}}

\def\mc{\mathcal}
\def\h{\mathcal{H}}

\def\mn{M^{(n)}}

\def\lesim{\lesssim}

\def\beq{\begin{equation}}
\def\endeq{\end{equation}}

\def\bs{\begin{split}}
\def\es{\end{split}}

\def\uk{u^{(k_0)}}

\def\pj{P^{(2)}}
\def\pi{P^{(1)}}
\def\pc{P^{\mathrm{Cone}}}
\def\a{\alpha}

\def\b{\beta}

\theoremstyle{plain}
\newtheorem{thm}{Theorem}[section]
\newtheorem{prop}[thm]{Proposition}

\newtheorem{lem}[thm]{Lemma}
\newtheorem{cor}[thm]{Corollary}

\newtheorem{claim}[thm]{Claim}

\newtheorem*{conj*}{Conjecture}
\newtheorem*{openproblem*}{Open Problem}
\newtheorem{openproblem}{Open Problem}

\theoremstyle{remark}
\newtheorem{rem}{Remark}[section]

\begin{document}

\title[Variable non-flat homogeneous curves]{Maximal operators and Hilbert transforms along variable non-flat homogeneous curves}

\author[S. Guo]{Shaoming Guo}
\address{Shaoming Guo: University of Bonn, Endenicher Allee 60, 53115 Bonn, Germany. Current Address: Indiana University Bloomington, 831 E Third St, Bloomington, IN 47405, USA}
\email{shaoguo@iu.edu}

\author[J. Hickman]{Jonathan Hickman}
\address{Jonathan Hickman: University of Chicago, 5734 S. University Avenue, Chicago, Illinois, 60637.}
\email{jehickman@uchicago.edu}

\author[V. Lie]{Victor Lie}
\address{Victor Lie: Department of Mathematics, Purdue, IN 46907 USA
\and Institute of Mathematics of the
Romanian Academy, Bucharest, RO 70700, P.O. Box 1-764, Romania.}
\email{vlie@purdue.edu}

\author[J. Roos]{Joris Roos}\address{Joris Roos: University of Bonn, Mathematical Institute, Endenicher Allee 60, 53115 Bonn, Germany}
\email{jroos@math.uni-bonn.de}

\date{\today}

\thanks{The third author was supported in part by NSF grant DMS-1500958. The fourth author is supported by the German National Academic Foundation.}

\begin{abstract}
 We prove that the maximal operator associated with variable homogeneous planar curves $(t, u t^{\alpha})_{t\in \mathbb{R}}$, $\alpha\not=1$ positive, is bounded on $L^p(\mathbb{R}^2)$ for each $p>1$, under the assumption that $u:\mathbb{R}^2 \to \mathbb{R}$ is a Lipschitz function. Furthermore, we prove that the Hilbert transform associated with $(t, ut^{\alpha})_{t\in \mathbb{R}}$, $\alpha\not=1$ positive, is bounded on $L^p(\mathbb{R}^2)$ for each $p>1$, under the assumption that $u:\mathbb{R}^2\to \mathbb{R}$ is a measurable function and is constant in the second variable. Our proofs rely on stationary phase methods, $TT^*$ arguments, local smoothing estimates and a pointwise estimate for taking averages along curves.
 \end{abstract}

\maketitle
\section{Introduction}

This paper focuses on the study of certain maximal and singular integral operators that act by integration along \textit{variable} homogeneous curves in the plane, of the form
$$\Gamma_{u}^{\alpha}(t):=(t,u\cdot [t]^\alpha)\,,$$
where here and throughout the paper, the exponent $\alpha$ is a fixed positive real number, the notation $[t]^\alpha$ stands for either $|t|^\alpha$ or $\mathrm{sgn}(t)|t|^\alpha$, while the ``coefficient" $u(\cdot,\,\cdot)$ is allowed to change depending on the base point $(x,y)\in\R^2$.

More precisely, given $u: \R^{2}\to \R$ a \textit{measurable} function and $0<\varepsilon_0\leq\infty$ a parameter, we consider the following objects:
\begin{itemize}

\item the ($\varepsilon_0$-truncated) \textbf{maximal operator along $\Gamma_{u}^{\alpha}$}, defined by
\beq\label{eqn:Maxdef}
\mc{M}^{(\alpha)}_{u,\varepsilon_0}f(x, y)=\sup_{0<\varepsilon<\varepsilon_0}\frac{1}{2\varepsilon}\int_{-\varepsilon}^{\varepsilon}|f(x-t, y-u(x, y)[t]^{\alpha})|dt.
\endeq

\item the ($\varepsilon_0$-truncated) \textbf{Hilbert transform along $\Gamma_{u}^{\alpha}$}, given by

\beq\label{eqn:HTdef}
\mc{H}^{(\alpha)}_{u,\varepsilon_0}f(x, y)=\textrm{p.v.}\:\int_{-\varepsilon_0}^{\varepsilon_0}  f(x-t, y-u(x, y)[t]^{\alpha})\frac{dt}{t}.
\endeq
\end{itemize}

For convenience, in what follows we will use the convention that $[t]^1=t$. Moreover, when $\alpha=1$, we will leave out the dependence on $\alpha$ and simply write $\mc{M}_{u, \epsilon_0}$ and $\mc{H}_{u,\varepsilon_0}$. The same principle applies to $\epsilon_0=\infty$: in this case we will simply write $\mc{M}^{(\alpha)}_u$ and $\mc{H}^{(\alpha)}_{u}$, respectively.\\

A difficult problem in the area of harmonic analysis is to understand the weakest possible regularity assumptions on $u$ that guarantee the $L^p$ boundedness of $\mc{M}^{(\alpha)}_{u,\varepsilon_0}$ and $\mc{H}^{(\alpha)}_{u,\varepsilon_0}$. Our aim in this paper is to provide a partial solution to this problem when we impose a nontrivial curvature condition by requiring $\alpha\not=1$.\\

We will now state our main results. The first result regards the boundedness of the maximal operator \eqref{eqn:Maxdef} and extends an earlier result of Marletta and Ricci \cite{MR}.\\
\newpage

\begin{thm}\label{thm:mainmax}
Let $\alpha>0$ and $\alpha\not=1$. Then the following hold:
\begin{enumerate}
\item \emph{\cite{MR}} If $u: \R^2\to \R$ is measurable, then for every $2<p\leq\infty$ we have
\beq\label{eqn:MaxLpbd1}
\|\mc{M}^{(\alpha)}_{u}f\|_p \le C_{p, \alpha}\|f\|_p\:.
\endeq
\item If $u: \R^2\to \R$ is Lipschitz, then there exists $\varepsilon_0=\varepsilon_0(\|u\|_{\mathrm{Lip}})>0$ such that for every $1<p\leq2$ we have
\beq\label{eqn:MaxLpbd2}
\|\mc{M}^{(\alpha)}_{u,\varepsilon_0}f\|_p \le C_{p, \alpha}\|f\|_p\:.
\endeq
\end{enumerate}
Here $C_{p, \alpha}\in (0,\infty)$ is a constant that depends only on $p$ and $\alpha$.
\end{thm}

The second main result regards the boundedness of the Hilbert transform \eqref{eqn:HTdef}. 

\begin{thm}\label{thm:mainHT}
Let $\alpha>0$ and $\alpha\not=1$. Let $u: \R^2\to \R$ be a measurable function and assume that
\beq\label{eqn:onevarassumpt}
u(x, y)=u(x,0) \text{ for every } x, y\in \R.
\endeq

Then we have that for all $1<p<\infty$ the following holds:
\beq\label{eqn:HTLpbd}
\|\mc{H}^{(\alpha)}_{u}f\|_p \le C_{p, \alpha}\|f\|_p.
\endeq
The constant $C_{p,\alpha}\in (0,\infty)$ depends only on $p$ and $\alpha$.
\end{thm}

\begin{rem}
 We would first like to stress that the analogue of estimate \eqref{eqn:MaxLpbd1} for the Hilbert transform $\mc{H}^{(\alpha)}_{u}$ fails for every $p\in (1,\infty)$ if we only assume $u$ to be measurable. This follows by a straight-forward modification of Karagulyan's \cite{Kara} construction of a counterexample in the case $\alpha=1$. However, if we assume $u$ to be Lipschitz it is possible that the analogue of \eqref{eqn:MaxLpbd2} for $\mc{H}^{(\alpha)}_{u}$ holds. Unfortunately, as of now, we are not able to prove or disprove this.
\end{rem}

\begin{rem}
Notice that unlike the situation described in Theorem \ref{thm:mainmax}, in Theorem \ref{thm:mainHT} we do not require any regularity assumptions on the function $u(x,0)$. Next, we remark that as opposed to \eqref{eqn:MaxLpbd2}, the $\epsilon_0$-truncation is not present in statement \eqref{eqn:HTLpbd}. This is a direct consequence of a standard scaling argument that makes the truncation in $\mc{H}^{(\alpha)}_{u, \epsilon_0}$ from \eqref{eqn:HTdef} become irrelevant. Nevertheless, the one-variable assumption \eqref{eqn:onevarassumpt} should be understood as being strictly stronger than the Lipschitz assumption imposed in Theorem \ref{thm:mainmax}. Indeed, we have the following corollary of Theorem \ref{thm:mainmax}.
\end{rem}

\begin{cor}\label{coro-max-one}
Under the same assumptions as in Theorem \ref{thm:mainHT}, we have
\beq\label{eqn:MaxOneVarLpbd} \|\mc{M}^{(\alpha)}_{u}f\|_p \le C_{p,\alpha} \|f\|_p \endeq
for all $1<p\le \infty$.
\end{cor}
The proof of Corollary \ref{coro-max-one} is via an anisotropic scaling argument which we sketch presently. First of all, by the scaling $x\to \lambda x, y\to \lambda^{\alpha} y,$ and a density argument, it suffices to show that
\beq\label{p1e1.7}
\|\mc{M}^{(\alpha)}_{u, \epsilon_0}f\|_p \le C_{p,\alpha} \|f\|_p,
\endeq
for all compactly supported smooth functions $f$. Here $\epsilon_0$ is the same as in Theorem \ref{thm:mainmax}. Now we approximate the chosen measurable function $u$ pointwisely by a sequence of Lipschitz functions $\{u_n\}_{n\in \N}$ satisfying $u_n(x,y)=u_n(x,0)$ for every $(x,y)\in\R^2$ and whose Lipschitz norms might grow to infinity. By changing variables $x\to x, y\to \lambda y$ in the $L^p$ integration on the left hand side of \eqref{eqn:MaxLpbd2}, we obtain \eqref{p1e1.7} with $u$ replaced by $u_n$, and with a constant independent of $n\in \N$. In the end, we apply the dominated convergence theorem to conclude \eqref{p1e1.7}.\\

The curvature condition is fundamental in the proofs of both Theorem \ref{thm:mainmax} and Theorem \ref{thm:mainHT}. Our approach relies on stationary phase methods, $TT^{*}-$arguments, local smoothing estimates and square function estimates. We speculate that the case $\alpha=1$ will require a combination of time-frequency techniques and methods presented in the current paper, but this subject remains open for future investigation.

In the following we present the historical evolution of the subject that motivated our interest for this study. We then continue with a discussion of our main results, embedding them in the historical context and also state some further results.

\subsection{Historical Background}

The historical landmark that generated much of the literature discussed below is the so-called Zygmund conjecture. This long-standing open problem  asks whether Lipschitz regularity of $u$ suffices to guarantee any non-trivial $L^p$ bounds for the maximal operator
\beq
\mc{M}_{u,\varepsilon_0} f(x,y) = \sup_{0<\varepsilon<\varepsilon_0} \frac1{2\varepsilon} \int_{-\varepsilon}^\varepsilon |f(x-t,y-u(x,y)t)| dt,
\endeq
provided $\varepsilon_0$ is small enough depending on $\|u\|_{\mathrm{Lip}}$. A counterexample based on a construction of the Besicovitch-Kakeya set shows that we cannot expect any $L^p$ bounds other than the trivial $L^\infty$ bound if $u$ is only assumed to be H\"older continuous with some exponent strictly smaller than one. The analogous question for $\mc{H}_{u,\varepsilon_0}$ is also widely open. For a detailed discussion of these conjectures, the interested reader is invited to consult Lacey and Li \cite{LL2}.\\

Notice that these two fundamental problems address the boundedness properties of \eqref{eqn:Maxdef} and \eqref{eqn:HTdef} along $\Gamma_{u}^{\alpha}$ in the following context:
\begin{itemize}
\item \textit{no regularity} above Lipschitz class is assumed;

\item \textit{no curvature} in the $t$ parameter is present, \textit{i.e.} $\alpha=1$, and thus these objects have rich classes of symmetries.
\end{itemize}

Partial progress towards understanding the above open problems developed relatively slowly in the past three decades; in direct relation with the itemization above it revolved around the nature of regularity and/or suitable curvature conditions imposed on the vector field $u$.

Bourgain \cite{Bo} proved that for every real analytic function $u$ there exists $\epsilon_0>0$ such that the associated maximal operator $\mc{M}_{u, \epsilon_0}$ is bounded on $L^2$. His argument can be extended to $L^p$ for all $p>1$ by using a suitable interpolation argument. For smooth vector fields, Christ, Nagel, Stein and Wainger \cite{CNSW} proved, under some extra curvature conditions, that the associated maximal operator and singular integral operators are bounded on $L^p$ for $p>1$.
The analogous result to that of Bourgain for singular integral operators was proved by Stein and Street \cite{SS}. Indeed, the result in \cite{SS} is far more general: in addition to the case of curves $(t, u\cdot t^{\alpha})$ with $u$ analytic and $\alpha\in \N$, they, in fact, consider all polynomials with analytic coefficients. \\

The first major breakthrough in terms of \textit{regularity} came when Lacey and Li \cite{LL1} brought tools from time-frequency analysis into the problem of Hilbert transforms along vector fields. To state these results, we introduce some notation.

Let $\psi_0:\R\to \R$ be a non-negative smooth function supported on the set $[-2, -1/2]\cup [1/2, 2]$ such that for all $t\not=0$,
\beq\label{cut-off-f}
\sum_{l\in \Z}\psi_l(t)=1\,,
\endeq
where $\psi_l(t):=\psi_0(2^{-l}t)$. For every $k\in \Z$, let $\pj_{k}$ denote the Littlewood-Paley projection in the $y$-variable corresponding to $\psi_k$. That is,
\beq
\pj_{k} f(x, y) := \int_{\R} f(x, y-\eta)\check{\psi}_k(\eta)d\eta.
\endeq
Similarly, we define $\pi_{k}$. Now we are ready to state the main result of Lacey and Li.

\begin{thm}[\cite{LL1}]\label{singleannulusline}
Let $u: \R^2\to \R$ be an arbitrary measurable function. For every $p\geq 2$ there exists $0<C_p<\infty$ such that the following hold:
\begin{itemize}
\item For all $k\in\Z$ we have
\beq\label{eqn:BTsingannL21}
\|\mc{H}_{u} \pj_{k} f\|_{2,\infty} \le C_2 \|\pj_{k} f\|_2.
\endeq
\item For all $p>2$ and $k\in\Z$ we have
\beq\label{eqn:BTsingannLp1}
\|\mc{H}_{u} \pj_{k} f\|_{p} \le C_p \|\pj_{k} f\|_p.
\endeq
\end{itemize}
\end{thm}

A few years later,  by further developing Lacey and Li's methods in \cite{LL1} and \cite{LL2},  Bateman \cite{Bateman} and  Bateman and Thiele \cite{BT} proved the following result.
\begin{thm}[\cite{Bateman}, \cite{BT}]\label{thm:BT}
Let $u: \R^2\to \R$ be a measurable function satisfying
\beq\label{eqn:onevarassumpt0}
u(x, y)=u(x,0)\:\:\:\:\:\textrm{a.e.}\:\:\:\:x, y\in \R.
\endeq
Then for every $1<p<\infty$ there exists $0<C_p<\infty$ such that
\beq\label{eqn:BTsingannLp}
\|\mc{H}_{u} \pj_{k} f\|_p \le C_p \|\pj_{k} f\|_p
\endeq
uniformly in $k\in \Z$. Moreover,  for all $p>3/2$, we have
\beq\label{eqn:BTLp}
\|\mc{H}_{u} f\|_p \le C_p \|f\|_p\:.
\endeq
\end{thm}
An earlier result on maximal operators and Hilbert transforms along one-variable vector fields can be found in Carbery, Seeger, Wainger and Wright \cite{CSWW}. An especially interesting aspect of that work is that they gave an endpoint result on the product Hardy space $H^1_{\text{prod}}(\R\times \R)$ under certain curvature assumption on the function $u$.\\

To pass from \eqref{eqn:BTsingannLp} to \eqref{eqn:BTLp}, Bateman and Thiele \cite{BT} relied crucially on the commutation relation
\beq\label{BTcommutation}
\mc{H}_{u} \pj_{k}=\pj_{k}  \mc{H}_{u}.
\endeq
Unfortunately, this relation fails for the maximal operator $\mc{M}_{u}$ as it is only a sub-linear operator. For this reason the following problem still remains open.
\begin{openproblem}\label{open1}
Let $u:\R^2\to \R$ be a measurable function satisfying \eqref{eqn:onevarassumpt0}. Is $\mc{M}_{u}$ bounded on $L^p(\R^2)$ for some $p<\infty$?
\end{openproblem}

This open problem served as our original motivation for the aim formulated at the beginning of our current paper. Further motivation for our study is provided by the rich literature addressing the boundedness of maximal and singular Radon transforms, focusing on curvature conditions. This corresponds to the case of $\Gamma_{u}^{\alpha}$ when $\alpha\not=1$.

Apart from the work \cite{MR} that has been mentioned already, one more relevant result from this body of literature is due to Seeger and Wainger \cite{SW03}. In this paper the variable curve $(t, u(x, y)\cdot [t]^{\alpha})_{t\in \R}$ appears as a special case of the more general curve $\Gamma(x, y, t)$ which satisfies some convexity and doubling hypothesis uniformly in $(x, y)$. For such curves $\Gamma(x, y, t)$, the authors proved that the associated maximal operator and singular integral operators are bounded on $L^p$ for $p>1$.

For more results in the same spirit, we refer to Nagel, Stein and Wainger \cite{NSW78}, Seeger \cite{See94} and the references therein.

\subsection{Comments on the main results}

We will start our discussion with Theorem \ref{thm:mainmax}.

As mentioned earlier, item (1) is the restatement of the result of Marletta and Ricci in \cite{MR}. To prove this result they used Bourgain's result on the circular maximal operator \cite{Bou86} as a black-box. In contrast, we will provide an alternative approach that is more self-contained. In a certain sense we are unraveling the mechanism behind the boundedness of the circular maximal operator, by using the local smoothing estimates of Mockenhaupt, Seeger and Sogge \cite{MSS92}  and the $l^2L^p$ decoupling inequalities for cones of Wolff \cite{Wolff00}, Bourgain \cite{Bou13} and Bourgain and Demeter \cite{BD15}. We do not claim any originality in this approach: that local smoothing estimates can be used to prove the boundedness of the circular maximal operator has already been pointed out in \cite{MSS92}. Moreover, the observation that decoupling inequalities for cones can provide certain progress toward the local smoothing conjecture is due to Wolff \cite{Wolff00}.

The next comment regards both items (1) and (2): recall that one of the main obstacles in the analysis of the maximal operator is that the analogue of the commutation relation \eqref{BTcommutation} fails due to sublinearity (see also \eqref{0103e1.18} below). Thus, in order to prove estimates \eqref{eqn:MaxLpbd1} and \eqref{eqn:MaxLpbd2} our strategy is to work with all frequency annuli at the same time and take advantage of the non-trivial curvature provided by $\Gamma_{u}^{\alpha}(t)=(t, u(x, 0)[t]^{\alpha})$ in the situation $\alpha\neq 1$.

The method of proving Theorem \ref{thm:mainmax} in the absence of the analogue of the commutation relation \eqref{BTcommutation} might also provide some insight toward Open Problem \ref{open1}.\\

We now focus our discussion on Theorem \ref{thm:mainHT}. This result can be regarded as the ``curved" analogue of \eqref{eqn:BTLp} from Bateman and Thiele's Theorem \ref{thm:BT}. In fact, in the next subsection we will state another result that includes \textit{the single annulus} version of both Theorem \ref{thm:mainHT} and Theorem \ref{thm:BT}, corresponding to \eqref{eqn:BTsingannLp} (see Theorem \ref{1602thm1.3} below). Moreover, in a forthcoming paper of the first, third and fourth author we will be relying in part on the ideas developed by the third author in \cite{Lie3}, \cite{Lie4} in order to extend Theorem \ref{thm:mainHT} to the setting of general curves (not necessarily homogeneous) obeying some suitable smoothness and curvature conditions.

Regarding the proof of Theorem \ref{thm:mainHT}, we rely on several ingredients. Following the general scheme in \cite{BT}, we first prove a single annulus estimate
\beq\label{1004e1.17}
\|\mc{H}^{(\alpha)}_{u} \pj_{k} f\|_p \lesim \|\pj_{k} f\|_p
\endeq
for all $p>1$ and then we use the commutation relation
\beq\label{0103e1.18}
\mc{H}^{(\alpha)}_{u} \pj_{k}=\pj_{k}  \mc{H}^{(\alpha)}_{u}
\endeq
to pass to a square function estimate. However, it is worth stressing here that the methods through which we achieve \eqref{1004e1.17} and then \eqref{eqn:HTLpbd} are quite different from the ones in \cite{BT}: there the authors use time-frequency techniques while in our case we rely on almost-orthogonality, stationary-phase and $T T^{*}$ methods derived from the presence of curvature.

This difference is also reflected in estimate \eqref{eqn:HTLpbd} from Theorem \ref{thm:mainHT} where we have an improved $L^p$ range including bounds for all $p$ close to 1. In contrast with this, the potential range for estimate \eqref{eqn:BTLp} in Theorem \ref{thm:BT} to hold is $p>4/3$. This exponent is related to the exponents in the variation norm Carleson theorem \cite{OSTTW12}. We refer to Bateman and Thiele \cite{BT} for a more detailed discussion.

In order to achieve the $L^p$ bounds for all $p>1$ we develop a pointwise estimate for taking averages along variable curves, via the shifted (strong) Hardy-Littlewood maximal function\footnote{See \eqref{0704e1.19} below.}. This pointwise estimate has a natural geometric interpretation: roughly speaking, it says that the averages along a thickened segment\footnote{That is, a small neighbourhood of a segment.} of the curve $(t, u\cdot |t|^{\alpha})$ can be pointwisely controlled, up to a small logarithmic loss, by a sum of averages taken over a number of rectangles. We remark that in the case $u\equiv 1$, our proof of Theorem \ref{thm:mainHT} reduces to an alternative proof for the $L^p$ boundedness of the classical singular Radon transform along the curve $(t,[t]^\alpha)$ for $p\not=2$, which does not seem to have appeared in the literature.\\

\subsection{Further results}
Here we present two more results.

As already mentioned above, the first result encompasses the single annulus estimates corresponding to Theorem \ref{thm:mainHT} and Theorem \ref{thm:BT}.

\begin{thm}\label{1602thm1.3}
Let $u: \R^2\to \R^2$ be a measurable function with $u(x, y)=(u_1(x, y), u_2(x, y))$. We then define the Hilbert transform along the variable polynomial curve
 $$(t, u_1(x, y)t+u_2(x, y)[t]^{\alpha})_{t\in \R}$$
by
\beq
\mc{H}^{(\alpha)}_{u} f(x ,y)=\emph{p.v.}\int_{\R}f(x-t, y-u_1(x, y)t-u_2(x, y)[t]^{\alpha})\frac{dt}{t}.
\endeq
Suppose now that
\beq
u(x, y)=u(x, 0)\:\:\:\:\:\:\textrm{a.e.}\:\:\:x, y\in \R\:.
\endeq
Then, for each $\alpha>0$ with $\alpha\neq 2$ and each $p>1$, there exists $C_{\alpha, p}>0$ such that
\beq\label{1702e1.6}
\|\mc{H}^{(\alpha)}_{u} \pj_{k} f\|_p \le C_{\alpha, p} \|\pj_{k} f\|_p,
\endeq
uniformly in $k\in \Z$, where here we recall that $\pj_{k}$ stands for the Littlewood-Paley projection operator in the $y$-variable.
\end{thm}

Notice that, by applying a partial Fourier transform in the $y$-variable and using Plancherel, our theorem implies the following.
\begin{cor}\label{poly-carleson}
For each $\alpha>0$ with $\alpha\neq 2$, we have
\beq\label{eqn:mockpolcarleson}
\left\|\sup_{u_1, u_2\in \R}\left| \emph{p.v.}\int_{\R}f(x-t)e^{iu_1 t+iu_2 [t]^{\alpha}}\frac{dt}{t} \right|\right\|_2 \le C_{\alpha} \|f\|_2,
\endeq
with a constant $C_{\alpha}$ depending only on $\alpha$.
\end{cor}

Remark that the case $\alpha=2$ is a deep result due to the third author \cite{Lie1} which Corollary \ref{poly-carleson} does not encompass due to the quadratic modulation symmetries present if $\alpha=2$.
The third author also proved bounds for the full polynomial Carleson operator \cite{Lie2},
\beq\label{eqn:polcarleson}
f\mapsto \sup_{P}\left|\textrm{p.v.}\int_\R f(x-t) e^{iP(t)} \frac{dt}t\right|,
\endeq
where the supremum goes over all polynomials $P$ with real coefficients of degree less than a fixed number. The proof uses a sophisticated time-frequency approach.
If $\alpha$ is a positive integer, \eqref{eqn:mockpolcarleson} is of course a corollary of the $L^2$ bounds for \eqref{eqn:polcarleson}. It is interesting however, that we can prove \eqref{eqn:mockpolcarleson} essentially using only Carleson's theorem as a black box.\\

Theorem \ref{1602thm1.3} is about a single annulus estimate. It can be viewed as an extension of Bateman's result \cite{Bateman}.  The proof of this result is a combination of the stationary phase method with an application of Bateman's single annulus estimate \cite{Bateman}.\\

Next, we will state a second single annulus estimate. It is the counterpart of Lacey and Li's result Theorem \ref{singleannulusline}, and will provide a key insight towards proving Theorem \ref{thm:mainmax}. However, it requires a completely different proof compared with Theorem \ref{1602thm1.3}.
\begin{thm}\label{thm:HTsingleannMeas}
Given $\alpha>0$ and $\alpha\neq 1$. For each measurable function $u: \R^2\to \R$, we have
\beq
\|H^{(\alpha)}_u \pj_{k} f\|_p \le C_{p, \alpha} \|\pj_{k} f\|_p,
\endeq
for each $p>2$. Here $C_{p, \alpha}$ does not depend on $k\in \Z$.
\end{thm}
The range $p>2$ is sharp in the sense that Theorem \ref{thm:HTsingleannMeas} fails for $p\le 2$. This can be seen by testing the estimate against the characteristic function of the unit ball.

The proof relies on the local smoothing estimates by Mockenhaupt, Seeger and Sogge \cite{MSS92}. The local smoothing estimates only work for $p>2$ which is reflected in the constraint $p>2$ in the above theorem. It is worth mentioning that we will not need the full strength of the local smoothing estimates, but only an ``$\varepsilon-$amount'' of them, and only for a single $p>2$. For this reason, we are able to provide a simple and self-contained proof of the local smoothing estimates we need, via decoupling inequalities for cones in $\R^3$. This will be the content of Appendix \ref{AppendixA}. We would like to emphasize again that the approach is due to Wolff \cite{Wolff00}. Moreover, in terms of the decoupling inequalities we use for cones in $\R^3$, we do not need the full range $2\le p\le 6$ in Bourgain and Demeter \cite{BD15}, but only the range $2\le p\le 4$. The decoupling inequalities for $p$ in this range again have a simple proof. For the sake of completeness, we include it here, see Appendix \ref{AppendixB}. This argument is due to Bourgain \cite{Bou13}.\\

{\bf Structure of the paper.}

\begin{itemize}

\item In Section \ref{sec:singleann} we prove Theorem \ref{1602thm1.3}. This is a single annulus version of the estimate in Theorem \ref{thm:mainHT}.

\item In Section \ref{sectionhilbertfull} we prove the full Theorem \ref{thm:mainHT}. The proof will rely on a vector-valued estimate for the shifted maximal operator.

\item In Section \ref{measurable-single} we show Theorem \ref{thm:HTsingleannMeas} whose proof serves as a preparation for the corresponding proof of Theorem \ref{thm:mainmax}.

\item  In Section \ref{section-main-theorem} we provide the proof of Theorem \ref{thm:mainmax}. Theorems \ref{thm:HTsingleannMeas} and \ref{thm:mainmax} rely on the local smoothing estimate in Theorem \ref{0813theorem8.1}.

\item In Appendix \ref{AppendixA} we provide the proof of  Theorem \ref{0813theorem8.1} via decoupling inequalities for cones in $\R^3$, following the approach of Wolff \cite{Wolff00}.

\item In Appendix \ref{AppendixB} we provide a proof of the decoupling inequalities we need, following the approach of Bourgain \cite{Bou13}.\\
\end{itemize}

{\bf Notation.} Throughout this paper, we will write $x\lesim y$ to mean that there exists a constant $C$ such that $x\le C y$. Here, unless otherwise specified, $C$ may only depend on fixed parameters, which will most commonly be $\alpha$ and $p$.
Further we write $x\sim y$ to mean that $x\lesim y$ and $y\lesim x$. Moreover, $x\approx y$ will mean $x\le 2y$ and $y\le 2x$. Lastly, $\mathbbm{1}_E$ will always denote the characteristic function of the set $E$.\\

{\bf Acknowledgements.} The first and fourth authors would like to thank Christoph Thiele for his guidance on the current project. They also thank Brian Street and Po-Lam Yung for useful discussions.\\

\section{A single annulus estimate}\label{sec:singleann}

In this section we prove Theorem \ref{1602thm1.3}. The special case $v\equiv 0$ will later be a key ingredient in the proof of Theorem \ref{thm:mainHT}.

Dropping the dependence on $u,v$ and $\alpha$ in our notation, we now set\footnote{Here and throughout the remainder of this text we will omit the principal value notation.}
\beq
\mc{H}f(x,y):= \int_\R f(x-t, y-v(x)t-u(x)[t]^\alpha) \frac{dt}{t}\,.
\endeq
Recall that throughout this section we always assume $\alpha\notin\{1, 2\}$. We intend to show that
\beq\label{1602e2.1}
\|\mc{H} \pj_{k} f\|_p \lesim \|\pj_{k} f\|_p,
\endeq
for each $p>1$ and  $k\in \Z$. The proof is a combination of the $TT^*$ method in the spirit of Stein and Wainger \cite{SW}, and the single annulus estimate for Hilbert transforms along one-variable vector fields by Bateman \cite{Bateman}. Here we will need a maximally truncated version of Bateman's result. \\

We start the proof of \eqref{1602e2.1}. By an anisotropic scaling
\beq
x\to x, y\to \lambda y,
\endeq
it suffices to prove \eqref{1602e2.1} for $k=0$. In the rest of this section we will always assume that $f=\pj_{0} f$. Furthermore, we assume without loss of generality that $u(x)> 0$ for almost every $x$. The case $u(x)<0$ can be handled similarly.\\

Observe that
\beq
\mc{H} f(x, y)=\sum_{l\in \Z}\int_{\R}f(x-t, y- v(x)t-u(x)[t]^{\alpha})\psi_l( u(x)^{1/\alpha}t) \frac{dt}{t}.
\endeq
Here $\psi_l$ is as defined in \eqref{cut-off-f}. Writing $\phi_0=\sum_{l\le 0}\psi_l$, we split the operator $\mc{H}$ into two parts:
\beq\label{1602e2.5}
\begin{split}
\mc{H} f(x, y)=&\int_{\R}f(x-t, y- v(x)t-u(x) [t]^{\alpha})\phi_0( u(x)^{1/\alpha}t) \frac{dt}{t}\\
& + \sum_{l\in \N} \int_{\R}f(x-t, y- v(x)t-u(x) [t]^{\alpha})\psi_l( u(x)^{1/\alpha}t) \frac{dt}{t}.
\end{split}
\endeq
We bound these two terms separately in the following two subsections.

\subsection{Low frequency part}\label{0814subsection3.1}

Here we treat the first summand on the right hand side of \eqref{1602e2.5}.
The idea is to compare it with the (maximally truncated) Hilbert transform along the one-variable vector field $(t, v(x)t)_{t\in \R}$ given by
\beq
\mc{\tilde{H}}^*f(x, y) := \int_{\R}f(x-t, y- v(x)t)\phi_0( u(x)^{1/\alpha}t) \frac{dt}{t}.
\endeq
We want the estimate
\beq\label{eqn:batemanblackbox}
\|\mc{\tilde{H}}^* f\|_p \lesim \| f\|_p
\endeq
to hold for all $p>1$. In the case $v\equiv 0$ this follows from the boundedness of the maximally truncated Hilbert transform. For an arbitrary $v$ it is a result essentially due to Bateman \cite{Bateman}. For a stronger variation norm estimate, see \cite{Guo2} by the first author. Now we look at the difference, which is given by
\beq\label{2111ee3.14}
\int_{\R} \left[ f(x-t, y- v(x)t-u(x) [t]^{\alpha})-f(x-t, y- v(x)t)\right]\phi_0( u(x)^{1/\alpha}t) \frac{dt}{t}.
\endeq
Recall that
\beq
f(x,y)=\int_{\R}f(x,y-z)\check{\psi}_0(z)dz.
\endeq
Substituting this identity into \eqref{2111ee3.14} we obtain
\beq\label{2111ee4.15}
\int_{\R}\int_{\R} f(x-t, y-v(x)t-z)\left[\check{\psi}_0(z-u(x)[t]^{\alpha})-\check{\psi}_0(z) \right]\phi_0( u(x)^{1/\alpha}t) \frac{dt}{t} dz.
\endeq
Using the key restriction  $|u(x)^{1/\alpha}t|\lesssim 1$ derived from \eqref{2111ee4.15}, we apply the fundamental theorem of calculus to deduce
\beq\label{2111ee4.16}
\left|\check{\psi_0}(z-u(x)[t]^{\alpha})-\check{\psi_0}(z) \right| \lesim \sum_{m\in \Z}\frac{1}{(|m|+1)^2} \mathbbm{1}_{[m, m+1]}(z)\cdot u(x)|t|^{\alpha}.
\endeq
Due to the sufficiently fast decay of $(|m|+1)^{-2}$, we will see that the summation in $m$ does not cause any problems.
For every $m\in\Z$ we consider the term
\beq\label{eqn:highfreqterm1}
\int_m^{m+1} \int_{\R} |f(x-t, y-v(x)t-z)| u(x)|t|^{\alpha}\phi_0( u(x)^{1/\alpha}t) \frac{dt}{|t|}dz
\endeq
arising from applying
\eqref{2111ee4.16} to \eqref{2111ee4.15} in the range $m < z < m+1$.
To bound this object in $L^p$ we will make use of the following simple observation.
\begin{lem}\label{lemma:minkowskitrick}
Let $\Gamma:\R^2\to\R$ be a measurable function and $1\le p\le\infty$.
If $K$ is a non-negative measurable function of two variables such that
\beq\label{eqn:minkowskiop1}
f\longmapsto \int_\R f(x-t) K(x,t) dt
\endeq
is bounded as an operator $L^p(\R)\to L^p(\R)$ with constant $C$,
then also
\beq\label{1004e2.14}
f\longmapsto \int_\R f(x-t, y-\Gamma(x,t)) K(x,t) dt
\endeq
is bounded as an operator $L^p(\R^2)\to L^p(\R^2)$ with constant $C$.
\end{lem}
\begin{proof}
Take the $L^p$ norm of the right hand side of \eqref{1004e2.14}. For fixed $x$ consider the quantity
\beq
\left(\int_\R \left|\int_\R f(x-t,y-\Gamma(x,t))K(x,t)dt\right|^p dy \right)^{1/p}.
\endeq
Applying Minkowski's integral inequality we bound this by
\beq\label{lemma 2.1 eqn}
\int_\R \|f(x-t,y-\Gamma(x,t))\|_{L^p(dy)} K(x,t) dt.
\endeq
Notice that $\Gamma(x ,t)$ is independent of $y$, hence by a simple change of variable, \eqref{lemma 2.1 eqn} is equal to
\beq
\int_\R \|f(x-t,y)\|_{L^p(dy)} K(x,t) dt.
\endeq
Now take the $L^p$ norm in $x$ and apply the hypothesis \eqref{eqn:minkowskiop1} to $g(x)=\|f(x,\cdot)\|_p$. This concludes the proof of Lemma \ref{lemma:minkowskitrick}.
\end{proof}

By Minkowski's inequality, the $L^p$ norm of \eqref{eqn:highfreqterm1} is no greater than
\beq\label{eqn:highfreqterm2}
\int_m^{m+1} \left(\int_{\R}\int_{\R}\left| \int_{\R} |f(x-t, y-v(x)t-z)| u(x)|t|^{\alpha}\phi_0( u(x)^{1/\alpha}t) \frac{dt}{|t|}\right|^p dydx\right)^{1/p} dz.
\endeq
Defining \beq
K(x,t) := u(x)|t|^\alpha \phi_0(u(x)^{1/\alpha} t) \frac1{|t|}
\endeq
we see that the operator in \eqref{eqn:minkowskiop1} is dominated by the Hardy-Littlewood maximal operator, which is bounded on $L^p(\R)$ for all $p>1$.
Hence \eqref{eqn:highfreqterm2} is bounded by (a constant multiple of) $\|f\|_p$. This completes the proof for the first summand on the right hand side of \eqref{1602e2.5}.

\subsection{High frequency part} Here we handle the second summand on the right hand side of \eqref{1602e2.5}.
By the triangle inequality it suffices to prove that there exists $\gamma>0$ such that
\beq\label{1602e2.18}
\left\| \int_{\R}f(x-t, y- v(x)t-u(x) [t]^{\alpha})\psi_l( u(x)^{1/\alpha}t) \frac{dt}{t} \right\|_p \lesim 2^{-\gamma l}\|f\|_p.
\endeq
By Lemma \ref{lemma:minkowskitrick}, \eqref{1602e2.18} holds for all $p>1$ without the exponentially decaying factor $2^{-\gamma l}$.

Hence by interpolation it suffices to prove \eqref{1602e2.18} for $p=2$. This will the goal of the present subsection. \\

To proceed, we apply a partial Fourier transform of the left hand side of \eqref{1602e2.18} in the $y$-variable. In view of Plancherel's theorem, \eqref{1602e2.18} for $p=2$ is equivalent to
\beq
\left\|\int_{\R} g_\eta(x-t) e^{i v(x)\eta t+i u(x)\eta [t]^{\alpha}}\psi_l(u(x)^{1/\alpha}t)\frac{dt}{t}\right\|_{L^2(dxd\eta)} \lesim 2^{-\gamma l}\|f\|_2,
\endeq
where
\beq
g_\eta(x) := \int_\R e^{-i\eta y} f(x,y) dy
\endeq
denotes the Fourier transform of the function $f$ in its second variable. By Fubini's theorem this reduces to proving the following one-dimensional estimate:
\beq\label{1602e2.21}
\left\|\int_{\R} g(x-t) e^{i v(x)t+i u(x)[t]^{\alpha}}\psi_l(u(x)^{1/\alpha}t)\frac{dt}{t}\right\|_{2} \lesim 2^{-\gamma l}\|g\|_2
\endeq
for all $g$ in $L^2(\R)$.

Observe that the phase function contains a linear term. This amounts to providing $L^2$ bounds for a modulation invariant operator with a special polynomial phase.

Given the shape of the operator one might guess that one has to use the time-frequency approach implemented by the third author in \cite{Lie1}, \cite{Lie2}. However, using crucially that $\alpha\not=1,2$, one can rely exclusively on $TT^*$ arguments in order to prove \eqref{1602e2.21}. A similar idea first appeared in a slightly simpler context in \cite{GPRY} by Pierce, Yung, the first author and the fourth author.\\

The first step to prove \eqref{1602e2.21} is to decompose the integral inside the norm on the left hand side of \eqref{1602e2.21} into regions where $t$ is either positive or negative. Both parts are treated in the same way, so we only detail the estimate for the positive part.
Accordingly, we denote
\beq
Tg(x) := \int_0^\infty g(x-t) e^{i v(x)t+i u(x)t^{\alpha}}\psi_l(u(x)^{1/\alpha}t)\frac{dt}{t}.
\endeq

Then we have
\beq
TT^*g(y)=\int_{\R}(\Phi^l_{u(y), v(y)}*\tilde{\Phi}^l_{u(x), v(x)})(y-x)g(x)dx,
\endeq
where here we set
\beq
\Phi^l_{u, v}(\xi):=e^{iv\xi+iu\xi^{\alpha}}\frac{\psi(2^{-l}u^{1/\alpha}\xi)}{\xi}
\quad\text{and}\quad \tilde{\Phi}^l_{u, v}(\xi):=\overline{\Phi^l_{u, v}(-\xi)}\,,
\endeq
with $\psi(\xi):=\psi_0(\xi) \chi_{(0,\infty)}(\xi)$. The kernel of $TT^*$ satisfies
\beq\label{1602e2.26}
\begin{split}
& |\Phi^l_{u(y), v(y)}*\tilde{\Phi}^l_{u(x), v(x)}|(\xi)\\
& =\left| \int_{\R} e^{i (v(y)-v(x))\eta+iu(y)\eta^{\alpha}-iu(x)(\eta-\xi)^{\alpha}} \frac{\psi(2^{-l} u(y)^{1/\alpha}\eta)}{\eta} \frac{\psi(2^{-l}u(x)^{1/\alpha}(\eta-\xi))}{\eta-\xi}d\eta\right|.
\end{split}
\endeq
Let us assume for the moment that $u(x)\le u(y)$.
Denoting
\beq
h:=\left(\frac{u(x)}{u(y)}\right)^{1/\alpha}\quad\text{and}\quad a:= 2^{-l}u(x)^{1/\alpha},
\endeq
via the change of variables
\beq
2^{-l} u(y)^{1/\alpha}\eta\to \eta,
\endeq
we see that \eqref{1602e2.26} equals to
\beq
a\left| \int_{\R} e^{i w\eta+i2^{\alpha l}\eta^{\alpha}-i2^{\alpha l}(h\eta-a\xi)^{\alpha}} \frac{\psi(\eta)}{\eta} \frac{\psi(h\eta-a\xi)}{h\eta-a\xi}d\eta\right|
\endeq
where $w$ is some quantity depending on $x,y$ and $l$, the value of which will be irrelevant to us. In the case $u(y)\le u(x)$ we interchange the roles of $u(x)$ and $u(y)$.

To finish this argument we use the following oscillatory integral estimate.
\begin{lem}\label{1602lemma2.1}
As above, assume $\alpha\notin\{1,2\}$ and let $\psi$ be a smooth function supported in $[1/2,2]$. Then, there exists $\lambda>0$ such that for all $\xi,w\in\R$, $0<h\le 1$ and $l>0$ we have
\beq\label{eqn:keyestimate}
\begin{split}
&  \left|  \int_{\R} e^{i w\eta+i2^{\alpha l}\eta^{\alpha}-i2^{\alpha l}(h\eta-\xi)^{\alpha}} \frac{\psi(\eta)}{\eta} \frac{\psi(h\eta-\xi)}{h\eta-\xi}d\eta\right|\\
& \lesim  \mathbbm{1}_{[-2^{-\lambda l}, 2^{-\lambda l}]}(\xi)+ 2^{-\lambda l}\mathbbm{1}_{[-2, 2]}(\xi).
\end{split}
\endeq
\end{lem}
We postpone the proof to the end of this section. Using the lemma we deduce that
\[ |\Phi^l_{u(y), v(y)}*\tilde{\Phi}^l_{u(x),v(x)}|(\xi)\lesssim \sum_{i=1,2} (a_i\mathbbm{1}_{[-2^{-\lambda l}, 2^{-\lambda l}]}(a_i \xi)+ 2^{-\lambda l} a_i\mathbbm{1}_{[-2, 2]}(a_i\xi)),  \]
where $a_1 := 2^{-l} u(y)^{1/\alpha}, a_2 := 2^{-l} u(x)^{1/\alpha}$.
Therefore we have
\beq
\begin{split}
|\langle TT^* g, h\rangle| &\le \int_\R \int_\R |(\Phi^l_{u(y),v(y)}*\tilde{\Phi}^l_{u(x),u(y)})(y-x) g(x) h(y) | dx dy\\
&\lesssim 2^{-\lambda l} \left( \int_\R\int_\R Mg(x) |h(y)| dx dy + \int_\R\int_\R |g(x)| Mh(y) dx dy \right)\\
&\lesssim 2^{-\lambda l} \|g\|_2 \|h\|_2.
\end{split}
\endeq
Here $M$ denotes the Hardy-Littlewood maximal function and we have used its $L^2$ boundedness as well as the Cauchy-Schwarz inequality in the last step.
This concludes the proof of \eqref{1602e2.21}.

\begin{proof}[Proof of Lemma \ref{1602lemma2.1}]
Denote the left hand side of \eqref{eqn:keyestimate} by $I_\xi$. First note that $I_\xi=0$ if $|\xi|>2$. Next, if $|\xi|\le 2^{-\lambda l}$, then the estimate follows from the triangle inequality and so we also assume that $|\xi|>2^{-\lambda l}$.
In the following we consider only $\eta$ such that the integrand in the integral defining $I_\xi$ is not zero. This implies that $\eta,h\eta-\xi\in [1/2,2]$.
We analyze the phase function
\beq
Q_\xi(\eta) := w\eta + 2^{\alpha l} (\eta^\alpha - (h\eta-\xi)^\alpha).
\endeq
Note that
\beq
Q_\xi''(\eta) = \alpha (\alpha-1) 2^{\alpha l} (\eta^{\alpha-2} - h^2(h\eta-\xi)^{\alpha-2}),
\endeq
\beq
Q_\xi'''(\eta) = \alpha (\alpha-1)(\alpha-2) 2^{\alpha l} (\eta^{\alpha-3} - h^3(h\eta-\xi)^{\alpha-3}),\\
\endeq

and observe that the vector $X: =\left(\begin{array}{c}
Q_\xi''(\eta)\\
Q_\xi'''(\eta)\end{array}\right)$ can be written as
\beq\label{eqn:matrixexpr}
\alpha(\alpha-1)2^{\alpha l}\left(\begin{array}{cc}
1 & 1\\
(\alpha-2)\eta^{-1} & (\alpha-2)h(h\eta-\xi)^{-1}\end{array}\right)
\left(\begin{array}{c}
\eta^{\alpha-2}\\
-h^2 (h\eta-\xi)^{\alpha-2}
\end{array}\right).
\endeq
This point of the argument crucially depends on the hypothesis $\alpha\not=2$.
Denoting the $2\times 2$ matrix in the above expression by $M$ we calculate
\beq
|\det(M)| = \frac{|(\alpha-2)\xi|}{(h\eta-\xi)\eta}\gtrsim |\xi|> 2^{-\lambda l}.
\endeq
This allows us to estimate
\beq
|X|\gtrsim 2^{(\alpha-\lambda)l}.
\endeq
Invoking van der Corput's lemma \cite[Chapter VIII.1]{Stein} we conclude that
\beq
I_\xi \lesssim 2^{-(\alpha-\lambda)/3 l}=2^{-\lambda l}\,,
\endeq
where in the last line above we have set $\lambda := \frac{\alpha}4$.
\end{proof}

\section{Proof of Theorem \ref{thm:mainHT}}\label{sectionhilbertfull}

Before starting our proof, for notational simplicity, we introduce the following convention: Throughout the section we omit the dependence on $u,\alpha$ and simply refer to our operator as $\mc{H}$. Recalling the commutation relation\footnote{Recall that $\pj_{k}$ denotes a Littlewood-Paley projection in the second variable.}
\beq\label{commutationrelation}
\h\pj_{k} =\pj_{k} \h\,,
\endeq
by Littlewood-Paley theory it suffices to prove that
\beq\label{eqn:s3squarefunct}
\Big\|\Big(\sum_{k\in \Z}|\mc{H}\pj_{k} f|^2\Big)^{1/2}\Big\|_p \lesim \|f\|_p\,.
\endeq

We stress here that the one-variable assumption \eqref{eqn:onevarassumpt} is the key fact that guarantees the commutation relation \eqref{commutationrelation}. This is the only place in this section where the one-variable assumption \eqref{eqn:onevarassumpt} is explicitly used. An implicit appearance is in the estimate \eqref{1702e4.15} for the case $p=2$, which is the content of the previous section.\\

We return to the proof of \eqref{eqn:s3squarefunct}. In Section \ref{sec:singleann} we already established that
\[ \|\mc{H} \pj_{k} f\|_p \lesssim \|\pj_{k} f\|_p \]
holds (this is the case $v\equiv 0$). Here we should note that the proof in Section \ref{sec:singleann} needs a small modification in the case when $v\equiv 0$ and $\alpha=2$. Namely, in that situation the exponential decay estimate \eqref{1602e2.18} is essentially a special case of a well-known result due to Stein and Wainger (see \cite[Theorem 1]{SW}).

Fix now $k\in\Z$. In view of the shape of the phase of our multiplier, we decompose our operator into a low and high frequency component respectively:
\beq
\begin{split}
\h \pj_{k} f(x, y)&=\sum_{l\in \Z}\int_{\R}(\pj_{k} f)(x-t, y- u(x) [t]^{\alpha})\psi_l( u(x)^{1/\alpha}t) \frac{dt}{t}\\
			&=\Big(\sum_{l\le -k/\alpha}+\sum_{l>-k/\alpha}\Big)\int_{\R}(\pj_{k} f)(x-t, y- u(x) [t]^{\alpha})\psi_l( u(x)^{1/\alpha}t) \frac{dt}{t}.
\end{split}
\endeq
We denote
\beq
T_{k, 0}f(x, y):=\sum_{l\le -k/\alpha}\int_{\R}f(x-t, y- u(x) [t]^{\alpha})\psi_l( u(x)^{1/\alpha}t) \frac{dt}{t}
\endeq
and, for $j\ge 1$
\beq
T_{k, j}f(x, y):=\int_{\R}f(x-t, y- u(x) [t]^{\alpha})\psi_{-\frac{k}{\alpha}+j}( u(x)^{1/\alpha}t) \frac{dt}{t}
\endeq
 Using the triangle inequality we obtain
\beq\label{1702e4.12}
\Big\|\Big(\sum_k |\h \pj_{k} f|^2\Big)^{1/2}\Big\|_p \lesim \Big\|\Big(\sum_k |T_{k, 0} \pj_{k} f|^2\Big)^{1/2}\Big\|_p+\sum_{j\in \N}\Big\|\Big(\sum_k |T_{k, j} \pj_{k} f|^2\Big)^{1/2}\Big\|_p.
\endeq
As in the previous section (see \eqref{2111ee3.14}  -- \eqref{eqn:highfreqterm2}) we treat $T_{k,0}\pj_{k} f$ as a perturbation of \beq
\sum_{l\le -k/\alpha}\int_{\R}P_k^{(2)} f(x-t, y)\psi_l( u(x)^{1/\alpha}t) \frac{dt}{t}.
\endeq
This yields
\beq\label{maximal domination}
|T_{k, 0}\pj_{k} f| \lesim M_S(\pj_{k} f)+ H^*(\pj_{k} f).
\endeq
Here $M_S$ denotes the strong maximal function and $H^*$ a maximally truncated Hilbert transform applied in the first variable. Indeed, one may deduce \eqref{maximal domination} using the same arguments as those of Section \ref{0814subsection3.1}. The vector-valued estimates for $M_S$ follow from the corresponding estimates for the one dimensional Hardy-Littlewood maximal function which are well-known (see Stein \cite[Chapter II.1]{Stein}). Similarly, the vector-valued estimates for $H^*$ follow from Cotlar's inequality and the vector-valued estimates for the Hilbert transform and the maximal function. Thus we have
\beq
\Big\|\Big(\sum_{k\in\Z} |T_{k, 0} \pj_{k} f|^2\Big)^{1/2}\Big\|_p \lesim \Big\|\Big(\sum_{k\in\Z}|\pj_{k} f|^2\Big)^{1/2}\Big\| \lesim \|f\|_p
\endeq
for all $p>1$. This finishes the proof for the first term on the right hand side of \eqref{1702e4.12}.\\

To bound the second term in \eqref{1702e4.12} we will prove that there exists $\gamma_p>0$ such that
\beq\label{1702e4.15}
\Big\|\Big(\sum_{k\in\Z} |T_{k, j} \pj_{k} f|^2\Big)^{1/2}\Big\|_p \lesim 2^{-\gamma_p j}\|f\|_p.
\endeq
For $p=2$ this follows from \eqref{1602e2.18} with $v\equiv 0$. Note here again that in the case $v\equiv0$, $\alpha=2$ the estimate \eqref{1602e2.18} is a consequence of \cite[Theorem 1]{SW}. Hence, by interpolation it suffices to prove
\beq\label{0704e1.9}
\Big\|\Big(\sum_{k\in\Z} |T_{k, j} \pj_{k} f|^2\Big)^{1/2}\Big\|_p \lesim j^4 \|f\|_p.
\endeq

Let us first note that if $p$ is sufficiently close to 2, then \eqref{0704e1.9} follows immediately from an interpolation argument. To carry out this interpolation, we observe the trivial pointwise bound
\beq
|T_{k, j} \pj_{k} f| \lesim 2^{\alpha j} M_S (\pj_{k} f).
\endeq
This implies that
\beq\label{1702e4.17}
\Big\|\Big(\sum_{k\in\Z} |T_{k, j} \pj_{k} f|^2\Big)^{1/2}\Big\|_p \lesim 2^{\alpha j}\|f\|_p
\endeq
for all $p>1$.
Interpolating with the bound for $p=2$ in \eqref{1702e4.15}, we can find a positive constant $\varepsilon_0>0$ such that \eqref{1702e4.15} holds true for all $p\in (2-\varepsilon_0, 2+\varepsilon_0)$.\\

Recall that our goal is to prove \eqref{0704e1.9} for all $p>1$.  For convenience we choose to present only the case $[t]^{\alpha}=|t|^{\alpha}$. The other case $[t]^{\alpha}=\mathrm{sgn}(t)|t|^{\alpha}$ can be treated by the same arguments.

Our strategy is to derive a sufficiently fine-grained pointwise estimate of $|T_{k,j} \pj_{k} f|$ by appropriate shifted maximal functions and then apply vector-valued bounds to conclude \eqref{0704e1.9}. The use of the shifted maximal operator was inspired by the work of the third author \cite{Lie4} (see there Section 2.4., Lemma 2).

First let us consider the case $k=0$. By definition we have
\beq
T_{0, j} \pj_{0} f(x, y)=\iint_{\R^2} f(x-t, y-s-u(x)|t|^{\alpha})\frac{\psi_{j}(u(x)^{\frac{1}{\alpha}}t)}t \check{\psi}(s) dtds.
\endeq
Up to Schwartz tails in $s$, this is essentially an average of $f$ over a thickened segment of a translate of the curve $(t,u(x)|t|^\alpha)$. The idea now is to cut up this thickened curve segment into pieces that are well approximated by rectangles.

Taking absolute values and using the triangle inequality we see that the previous display is
\beq
\lesssim \frac{1}{\lambda_{x,j}}\int_{|t|\approx\lambda_{x,j}} \int_\R |f(x-t, y-s-u(x) |t|^{\alpha})\check{\psi}(s)|ds dt,
\endeq
where $\lambda_{x,j} := 2^j u(x)^{-1/\alpha}$ and the notation $|t|\approx \lambda$ means $\frac12 \lambda\le |t|\le 2\lambda$.
By the rapid decay of $\check{\psi}$ this is
\beq\label{1004e4.23}
\lesssim \sum_{\tau\in \Z}\frac{1}{(1+|\tau|)^{10}}\frac{1}{\lambda_{x,j}}\int_{|t|\approx\lambda_{x,j}} \int_{\tau}^{\tau+1} |f(x-t, y-s-u(x) |t|^{\alpha})|ds dt.
\endeq

Once at this point, the intuition is given by the following observation: the function $f= \pj_{0} f$ ``sees" the  $y-$universe in unit steps; that is, $f$ is morally $y-$constant on segments of length one. Consequently, it is natural to further discretize the location of $t$ in intervals on which the variation of the term $u(x) |t|^{\alpha}$ does not exceed the order of one. This invites us to consider the following construction.

Set $\delta_{x,j} := 2^{-(\alpha-1)j} u(x)^{-1/\alpha}$ and cover the region $\frac12 \lambda_{x,j}\le |t|\le 2\lambda_{x,j}$ by
intervals $\{I_m\}_{m=0}^{N_j-1}$ where
\[I_m=\Big\{t\,:\,\frac12 \lambda_{x,j}+m\delta_{x,j}\le |t|\le \frac12\lambda_{x,j}+(m+1)\delta_{x,j}\Big\}\]
and $N_{j}\in\N$ is such that $\frac32 \lambda_{x,j}\le N_{j}\delta_{x,j}\le 2\lambda_{x,j}$.  Notice that $N_{j}~\sim ~2^{\alpha j}$ and moreover that $N_j$ can be chosen independently of $x$.

With this we have
\[ \frac{1}{\lambda_{x,j}}\int_{|t|\approx\lambda_{x,j}} \int_{\tau}^{\tau+1} |f(x-t, y-s-u(x) |t|^{\alpha})|ds dt \]
\beq\label{eqn:HTpflargepcutting1} \lesssim \frac1{N_{j}} \sum_{m=0}^{N_{j}-1} \frac1{|I_m|} \int_{I_m} \int_0^1 |f(x-t,y-s-\tau-u(x)|t|^\alpha)| ds dt. \endeq
We now set
\[ \mathbf{C}_m := \Big\{ (t,s+\tau+u(x)|t|^\alpha)\,:\, t\in I_m,\,s\in[0,1] \Big\}\subset I_m\times J_m,\]
where
\[ J_m = \left[\tau+u(x)\Big(\frac12\lambda_{x,j}+m\delta_{x,j}\Big)^\alpha, \tau+1+u(x)\Big(\frac12\lambda_{x,j}+(m+1)\delta_{x,j}\Big)^\alpha\right]. \]
 Notice that, because of our choice of $\delta_{x,j}$, the thickened curve segment $\mathbf{C}_m$ is contained in a rectangular region of comparable area. Indeed, we have by the mean value theorem,
\[ |J_m|\sim_\alpha 1+u(x)\delta_{x,j} \lambda_{x,j}^{\alpha-1}\sim_\alpha 1. \]
Thus we further have that \eqref{eqn:HTpflargepcutting1} is bounded by a constant multiple of
\[\frac1{N_j}\sum_{m=0}^{N_j-1}\frac1{|I_m\times J_m|} \iint_{I_m\times J_m} |f(x-t,y-s)| dtds. \]
Given a non-negative parameter $\sigma$, we define the \emph{shifted maximal operator} as
\beq
M^{(\sigma)} f(z):=\sup_{z\in I\subset\R}\frac{1}{|I|} \int_{I^{(\sigma)}} |f(\zeta)|d\zeta.
\endeq
Here the supremum goes over all bounded intervals $I$ containing $z$, and $I^{(\sigma)}$ denotes a shift of the interval $I=[a,b]$ given by
\[I^{(\sigma)}:=[a-\sigma\cdot|I|, b-\sigma\cdot |I|]\cup [a+\sigma|I|, b+\sigma|I|]. \]
Note that
\[ \frac1{|I_m\times J_m|} \iint_{I_m\times J_m} |f(x-t,y-s)| dtds \le M_1^{(\sigma^{(1)}_m)} M_2^{(\sigma^{(2)}_m+\tau)} f(x,y), \]
where
\[ \left\{\begin{array}{ll}
\sigma^{(1)}_m &:= 2^{\alpha j-1}+m,\\
\sigma^{(2)}_m &:= c_\alpha (2^j+2^{-(\alpha-1)j} m)^\alpha,
\end{array}\right. \]
and $c_\alpha$ is a constant only depending on $\alpha$ and  $M_1^{(n)}$ (respectively, $M_2^{(n)}$) denotes the shifted maximal operator applied in the first (respectively, second) variable. Notice that since $N_{j}\sim 2^{\alpha j}$ and $m<N_{j}$ we have that $\sigma^{(i)}_m\lesssim 2^{\alpha j}$ for $i=1,2$.

Altogether we have now proved that
\[ |T_{0,j}\pj_{0} f(x,y)| \lesssim \sum_{\tau\in \Z} \frac{1}{(1+|\tau|)^{10}}\frac1{N_{j}} \sum_{m=0}^{N_{j}-1} M_1^{(\sigma_m^{(1)})} M_2^{(\sigma_m^{(2)}+\tau)} f(x,y). \]
By a scaling argument

 we have that for all $k\in\Z$ the following holds:
\beq\label{0704e1.19}
|T_{k,j}\pj_{k} f(x,y)| \lesssim \sum_{\tau\in \Z} \frac{1}{(1+|\tau|)^{10}}\frac1{N_{j}} \sum_{m=0}^{N_{j}-1} M_1^{(\sigma_m^{(1)})} M_2^{(\sigma_m^{(2)}+\tau)} \pj_{k} f(x,y). \endeq
Inserting these two bounds into the left hand side of \eqref{0704e1.9} yields
\beq
\Big\|\Big(\sum_{k\in\Z}\Big(\sum_{\tau\in \Z} \frac{1}{(1+|\tau|)^{10}}\frac{1}{N_j}\sum_{m=0}^{N_j-1} M_1^{(\sigma^{(1)}_m)} M_2^{(\sigma^{(2)}_m+\tau)}\pj_{k} f(x, y)\Big)^2\Big)^{1/2}\Big\|_p.
\endeq
By the triangle inequality this is no greater than
\beq
\sum_{\tau\in \Z} \frac{1}{(1+|\tau|)^{10}}\frac1{N_j}\sum_{m=0}^{N_j-1} \Big\|\Big(\sum_{k\in\Z}  \Big(M_1^{(\sigma^{(1)}_m)} M_2^{(\sigma^{(2)}_m+\tau)}\pj_{k} f(x, y)\Big)^2\Big)^{1/2}\Big\|_p.
\endeq
Thus, to show \eqref{0704e1.9} it suffices to prove that
\beq\label{ef}
\Big\|\Big(\sum_{k\in\Z} \Big(M_1^{(\sigma^{(1)}_m)} M_2^{(\sigma^{(2)}_m+\tau)}\pj_{k} f(x, y)\Big)^2\Big)^{1/2}\Big\|_p \lesim j^4\cdot \log^2(2+|\tau|) \|f\|_p.
\endeq
Since $\sigma^{(i)}_m\lesssim 2^{\alpha j}$ for $i=1,2$, by Fubini's theorem we have that \eqref{ef} is a consequence of the following vector-valued estimate for the one-dimensional shifted maximal operator:
\beq\label{0704e1.23}
\left\|\Big( \sum_{k\in \Z} |\mn f_k|^2\Big)^{1/2}\right\|_p \lesim (\log\langle n\rangle)^2 \Big\|\Big(\sum_{k\in \Z}|f_k|^2\Big)^{1/2}\Big\|_p,
\endeq
where we adopt the Japanese bracket notation $\langle n\rangle :=2+|n|$. We give the proof of this last statement below.\\

Let $\mathcal{D}$ denote the set of dyadic intervals $I=[2^k m, 2^k (m+1))$ with $k,m\in\Z$. In accordance with the above definition we have,
$$I^{(n)} = [2^k (m-n), 2^k (m-n+1))\in\mathcal{D}$$
for $n\in\Z$. For simplicity we discuss only the dyadic variant of $M^{(n)}$, defined as
\[ M^{(n)} f(x) = \sup_{x\in I\in\mathcal{D}} \frac{1}{|I|} \int_{I^{(n)}} |f(y)| dy. \]
Everything here carries over to the non-dyadic version with the constants having the same dependence on $n$ (see Muscalu \cite[p. 741]{Mu}).

\begin{thm}\label{feff-stein-shifted}
For $1<p<\infty, 1<q\le\infty$ we have
\beq \Big\|\Big(\sum_{k\in\Z} |M^{(n)} f_k|^q\Big)^{1/q}\Big\|_p \lesssim (\log \langle n\rangle)^2 \Big\|\Big(\sum_{k\in\Z} |f_k|^q \Big)^{1/q}\Big\|_p.\endeq
\end{thm}

We did not find a reference for this result in the literature, so we choose to provide a short self-contained proof. The scalar version of this estimate involves standard Calder\'on-Zygmund techniques and can be found in \cite{Mu}, where the author attributes it to Stein \cite{Stein}.

Before starting the proof we make few more observations: firstly, notice that the endpoints $1<p=q\le \infty$ and $q=\infty$ of Theorem \ref{feff-stein-shifted} follow immediately from the scalar version, and thus interpolation establishes the result for $1<p\le q<\infty$. Secondly, the exponent $2$ for the log loss is only chosen for convenience; the proof actually gives a slightly better exponent.

Now the proof we present below relies on a weighted estimate in the spirit of Fefferman-Stein \cite{FeffermanStein}:
\begin{lem}\label{lemma:weightedest}
Let $\omega\ge 0$ be a locally integrable function. For all $\lambda>0$,
\beq \omega(\{x\,:\,M^{(n)} f(x)>\lambda\})\lesssim \lambda^{-1}\Big(\log\langle n\rangle \|f\|_{L^1(M^{(-n)}\omega)} + \|f\|_{L^1(M\omega)}\Big),\endeq
where $M$ denotes the Hardy-Littlewood maximal function.
\end{lem}

\begin{proof}
Fix $\lambda>0$. Let $\mathcal{I}$ be the collection of maximal dyadic intervals $I$ such that
\beq \frac1{|I|} \int_{I} |f|>\lambda.\endeq
Given $I\in\mathcal{I}$ we denote by $\mathcal{J}^I_i$ the collection of dyadic intervals $J$ such that $|J|=2^{-i} |I|$, $J^{(n)}\subset I$ and $\frac1{|J|}\int_{J^{(n)}} |f|>\lambda.$ This is the $i$-th generation of shifted subintervals of $I$. We call $i$ \emph{large} if $2^{-i} |n|<1$ and otherwise we call $i$ \emph{small}. It will be important that this depends only on $n$. Denote $\mathcal{J}^I=\bigcup_{i\ge 0} \mathcal{J}^I_i$.
Observe that
\beq\{ x\,:\,M^{(n)}f(x)>\lambda\} \subset \bigcup_{I\in\mathcal{I}} \bigcup_{J\in \mathcal{J}^I} J. \endeq
Indeed, if $x$ is such that $M^{(n)} f(x)>\lambda$ then there exists $J\in\mathcal{D}$ with $x\in J$ and $\frac1{|J|} \int_{J^{(n)}} |f|>\lambda$. By definition of $\mathcal{I}$ there is some $I\in\mathcal{I}$ with $J^{(n)}\subset I$ and therefore $x\in J\in \mathcal{J}^I$.
The crucial observation is that if $J\in \mathcal{J}^I_i$ and $i$ is large, then $J$ is contained in $3I$.

Thus we can estimate
\begin{align}\label{eqn:weightedestpf1}
\omega(\{x\,:\,M^{(n)}f(x)>\lambda\}) &\le \sum_{I\in\mathcal{I}} \Big( \sum_{\substack{J\in\mathcal{J}^I_i,\\i\;\text{small}}} \int_J \omega + \int_{3I} \omega \Big).
\end{align}
For $J\in\mathcal{J}^I_i$, since $(J^{(n)})^{(-n)}=J$ we have that
\beq \int_J \omega \le \frac1\lambda \int_{J^{(n)}} |f(x)|\cdot\left( \frac1{|J|}\int_{J} \omega(y) dy\right) dx \le \frac1\lambda \int_{J^{(n)}} |f(x)| M^{(-n)}\omega(x) dx. \endeq
Similarly,
\beq \int_{3I} \omega \lesssim \frac1\lambda \int_I |f(x)| M\omega(x) dx. \endeq
Thus we can estimate \eqref{eqn:weightedestpf1} by
\beq \lambda^{-1}\Big(\sum_{I\in\mathcal{I}} \sum_{\substack{J\in\mathcal{J}^I_i,\\i\;\text{small}}} \int_{J^{(n)}} |f| M^{(-n)}\omega + \sum_{I\in\mathcal{I}} \int_I |f| M\omega\Big). \endeq
For fixed $i$ the $J\in \mathcal{J}_i^I$ are disjoint and we have $\bigcup_{J\in\mathcal{J}_i^I} J^{(n)}\subset I$. Since there are about $\log\langle n\rangle$ small $i$, the previous display is bounded by
\beq \lambda^{-1}\left( \log\langle n\rangle\|f\|_{L^1(M^{(-n)}\omega)} + \|f\|_{L^1(M\omega)}\right). \endeq
This finishes the proof of Lemma \ref{lemma:weightedest}.
\end{proof}

Now let us treat the range $p\ge q$. Lemma \ref{lemma:weightedest} says that $M^{(n)}$ is a bounded operator $L^1(\widetilde{M}^{(-n)}\omega)\to L^{1,\infty}(\omega)$ with constant $\lesssim \log\langle n\rangle$ where $\widetilde{M}^{(n)}=M^{(n)}+M$. By interpolation with the trivial $L^\infty$ bound we get
\beq \int (M^{(n)} f_k)^q \omega \lesssim \log\langle n\rangle \int |f_k|^q \widetilde{M}^{(-n)} \omega \endeq
for all $1<q<\infty$. Let $r$ be the dual exponent of $p/q$. We have
\beq \Big\|\Big(\sum_{k\in\Z} |M^{(n)} f_k|^q\Big)^{1/q}\Big\|_p^q = \Big\|\sum_{k\in\Z} |M^{(n)} f_k|^q\Big\|_{p/q} = \sup_{\|\omega\|_r\le 1} \int \sum_{k\in\Z} (M^{(n)} f_k)^q \omega. \endeq
Using the previous estimate we bound this by
\beq \log\langle n\rangle \int \sum_{k\in\Z} |f_k|^q \widetilde{M}^{(-n)} \omega. \endeq
By H\"older's inequality and the scalar logarithmic bound for $M^{(n)}$ we have
\beq\int \sum_{k\in\Z} |f_k|^q \widetilde{M}^{(-n)} \omega\le \Big\|\Big(\sum_{k\in\Z} |f_k|^q\Big)^{\frac1{q}}\Big\|^q_p \|\widetilde{M}^{(-n)} \omega\|_r\lesssim (\log\langle n\rangle)^{\frac1{r}} \Big\|\Big(\sum_{k\in\Z} |f_k|^q\Big)^{\frac1{q}}\Big\|^q_p. \endeq
This finishes the proof of our theorem. 

\section{Proof of Theorem \ref{thm:HTsingleannMeas}}\label{measurable-single}

In this section we prove Theorem \ref{thm:HTsingleannMeas}. The main tool we will be using is the local smoothing estimate from Theorem \ref{0813theorem8.1}.\\

We start the proof by recalling that
\beq
\mc{H}^{(\alpha)}_{u}f(x, y) := \int_{\R} f(x-t, y-u(x, y)[t]^{\alpha})\frac{dt}{t}.
\endeq
We will only present the proof of the case $\alpha>1$; the remaining case $0<\alpha<1$ can be treated using the same methods and is somewhat easier. As $\alpha$ and $u$ will always be fixed, we will leave out the dependence on them in our notation and simply use $T$ to denote $\mc{H}_u^{(\alpha)}$.  In this section, we will prove
\beq\label{2704e2.2}
\|T \pj_{k} f\|_p \lesim \|\pj_{k} f\|_p,
\endeq
for all $p>2$ and all measurable function $u: \R^2\to \R$, with a bound independent of $k\in \Z$ and $u$. By the anisotropic scaling
\beq
x\to x, y\to \lambda y,
\endeq
it suffices to prove \eqref{2704e2.2} for $k=0$. \\

In order to simplify our presentation, we introduce some notation. We let
\beq\label{08300e5.4}
z=(x,y)\quad\text{and}\quad u_z := u(x,y),
\endeq
and set $v_z$ to be the unique integer such that
\beq\label{08300e5.5}
2^{v_z}\le u_z< 2^{v_z+1}.
\endeq
For a given $k_0\in \Z$, define
\beq\label{08300e5.6}
u^{(k_0)}_z := 2^{k_0-v_z} u_z.
\endeq
Observe that $u^{(v_z)}_z = u_z$, $u^{(k_0)}_z\in [2^{k_0}, 2^{k_0+1})$ and $u^{(k_0)}_z = 2^{k_0} u^{(0)}_z$.

Denote
\beq
T_{k_0} f(x, y):=\int_{\R} f(x-t, y- \uk_z [t]^{\alpha})\frac{dt}{t}.
\endeq
\begin{rem}
Roughly speaking, we will bound $T \pj_{0} f$ by the ``square function''
\beq\label{0814e6.6h}
\Big(\sum_{k_0\in \Z} |T_{k_0}\pj_{0} f|^p\Big)^{1/p}.
\endeq
Here we are using an $l^p$ sum instead of an $l^2$ sum. This is because $p$ is always larger than two. At first glance, it might be a bit surprising that the term \eqref{0814e6.6h} is still a bounded operator. The proof of this fact will be achieved by applying a finer decomposition for the function $f$, and then seeking for enough ``off-diagonal'' decay via a local smoothing estimate.
\end{rem}

We now begin our analysis by performing a dyadic decomposition of the kernel $\frac{1}{t}$ around the singularity $t=0$. In particular, let
\beq\label{0814e6.7h}
T_{k_0, l} f(x, y):=\int_{\R} f(x-t, y- \uk_z [t]^{\alpha})\psi_l((u^{(0)}_z)^{\beta}t)\frac{dt}{t},
\endeq
where $\beta=\frac{1}{\alpha-1}$. Recall that $\alpha>1$ and so $\beta$ is always positive. The motivation for using the factor $(u^{(0)}_z)^{\beta}$ and the choice of $\beta$ will become clear much later during the main argument (see the proof of Lemma \ref{localsmoothinglemma}, specifically \eqref{eqn:localsmoothinglemmapf3}).

We next perform a Littlewood-Paley decomposition in the $x$-variable, and write
\beq\label{2704e2.8}
T_{k_0} \pj_{0} f=\sum_{l\in \Z} \sum_{k\in \Z} T_{k_0, l} \pi_k\pj_0 f.
\endeq

We will split the sums in $l, k\in \Z$ in two major cases according to the behavior of the phase of the ``multiplier'' (see the discussion preceding \eqref{0304e5.11} below) associated with $T_{k_0, l}$:
\begin{enumerate}
\item \textbf{low frequency case}: $l< \max\{-k, -k_0/\alpha\}$;

In this situation the operator $T_{k_0, l}$ behaves like a one-dimensional convolution operator.

\item \textbf{high frequency case}: $l\geq \max\{-k, -k_0/\alpha\}$;

This case splits in two subcases:
\begin{enumerate}
\item $l$ sits below the critical point of the phase;
\item $l$ sits above the critical point of the phase.
\end{enumerate}

\end{enumerate}
In what follows we explain the heuristic for the above partition of our analysis.

Fix $l, k\in \Z$ and focus on the function $T_{k_0, l} \pi_k f$.
Imagine for the moment that the function $u_z^{(k_0)}$ is a constant $u^{(k_0)}$ on the interval $[2^{k_0}, 2^{k_0+1})$. Then $T_{k_0, l} \pi_{k}$ becomes a convolution type operator, and hence it makes sense to speak about its multiplier as
\beq\label{0304e5.11}
\int_{\R} e^{i 2^l t\xi+ i \uk 2^{\alpha l} [t]^{\alpha} \eta} \psi_0([u^{(0)}_z]^{\beta}t)\frac{dt}{t}.
\endeq

Recall that $\xi\sim 2^k$ and $\eta\sim 1$. In the situation described by item (1) either $2^l \xi\lesim 1$ or $u^{(k_0)} 2^{\alpha l}\eta\lesim 1$. When one of these two inequalities occurs, say $u^{(k_0)} 2^{\alpha l}\eta\lesim 1$, then in the expression \eqref{0814e6.7h}, we can view $f(x-t, y-u^{(k_0)} [t]^{\alpha})$ as a perturbation of $f(x-t, y)$. This is what we meant when saying that the operator $T_{k_0, l}$ behaves like a one-dimensional operator.

Assume now we are in the situation of item (2), that is $l\geq \max\{-k, -k_0/\alpha\}$. In this case we have two possibilities:
\beq
k< k_0/\alpha \text{ and } k\geq k_0/\alpha\,.
\endeq
In the first instance, $k< k_0/\alpha$, the phase function in \eqref{0304e5.11} does not admit any critical point, which makes this case much easier to handle. This is the reason for which we will only focus on the latter case $k\geq k_0/\alpha$.

Now, as explained above, the cutoff between case (2a) and (2b) is indicated by the stationary phase principle. Analyzing the stationary points we deduce the requirement
\beq
2^{l+k}\sim 2^{\alpha l+k_0} \implies |l-\beta\cdot (k-k_0)|\lesssim 1\,.
\endeq
Once at this point, we let the situation in item (2a) be defined by the conditions:
\beq\label{hl}
k\geq k_0/\alpha\:\:\textrm{and}\:\;l<\beta\cdot (k-k_0)\,.
\endeq
while the situation in item (2b) be defined by the conditions:
\beq\label{hh}
k\geq k_0/\alpha\:\:\textrm{and}\:\;l\geq\beta\cdot (k-k_0)\,.
\endeq
With this heuristic, based on items (1) and (2) above we split \eqref{2704e2.8} as
\begin{align}\label{0304e5.12}
T_{k_0}  \pj_{0}f &=\sum_{k\in \Z}\sum_{l\le \max\{-k, -k_0/\alpha\}} T_{k_0, l} \pi_{k} \pj_{0}f+ \sum_{k\in \Z}\sum_{l>\max\{-k, -k_0/\alpha\}} T_{k_0, l} \pi_{k} \pj_{0}f \\
\nonumber
&=: I_{k_0}+II_{k_0}.
\end{align}
These two terms are treated separately in the following two subsections.

\subsection{The high frequency case}
In this subsection, we will treat the term $II_{k_0}$ which is, in fact, the main term in \eqref{0304e5.12}.

Recall from our heuristic that in the situation $k<k_0/\alpha$ the phase function in \eqref{0304e5.11} does not admit a critical point. That makes this case easier to handle. The precise arguments are the same as for the main term, so we will not detail them here.
That is, we will only treat the case $k\geq k_0/\alpha$. Accordingly we redefine
\beq\label{0304e5.16}
II_{k_0}:=\sum_{k\ge k_0/\alpha}\:\sum_{l> -k_0/\alpha} T_{k_0, l} \pi_{k} \pj_{0} f.
\endeq

As $\alpha>1$, in this situation we always have
$$-k_0/\alpha\leq \beta\cdot (k-k_0).$$
Thus we split $II_{k_0}$ into two parts
\begin{align}\label{0814e6.16h}
II_{k_0}&=\sum_{k\ge \frac{k_0}{\alpha}}\sum_{l> -k_0/\alpha}^{\beta\cdot (k-k_0)}T_{k_0, l} \pi_{k} \pj_{0} f +\sum_{k\ge \frac{k_0}{\alpha}}\sum_{l>\beta\cdot (k-k_0)}T_{k_0, l} \pi_{k} \pj_{0} f \\
\nonumber
&=: II_{k_0}^{(1)}+II_{k_0}^{(2)}\:.
\end{align}

Our goal here will be to prove that
\beq\label{goal2}
\Big\|\Big(\sum_{k_0\in \Z}|II_{k_0}^{(j)}|^p\Big)^{1/p}\Big\|_p \lesim \|f\|_p\:\:\:\:\:\:\:\textrm{for}\:j\in\{1,\,2\}\,.
\endeq
For this, we first apply the change of variables\footnote{Note that $\beta(k-k_0)$ might not be an integer but this is irrelevant.} $l\:\rightarrow\:l+\beta(k-k_0)\:,$
and use Fubini to deduce that
\beq\label{21ter}
II_{k_0}^{(1)}=\sum_{l=-\infty}^{0} \sum_{k>\frac{k_0}{\alpha}-\frac{l}{\b}}T_{k_0, \b(k-k_0)+l} \pi_{k} \pj_{0} f\,,
\endeq
and
\beq\label{22ter}
II_{k_0}^{(2)}=\sum_{l=1}^{\infty} \sum_{k\geq\frac{k_0}{\alpha}}T_{k_0, \b(k-k_0)+l} \pi_{k} \pj_{0} f\:.
\endeq
In order to prove \eqref{goal2} it will be sufficient to show that for all $p>2$ we have
\beq\label{21}
\Big\|\Big(\sum_{k_0\in \Z}|\sum_{k>\frac{k_0}{\alpha}-\frac{l}{\b}} T_{k_0, \beta\cdot (k-k_0)+l} \pi_{k} \pj_{0}f|^p\Big)^{1/p}\Big\|_p \lesim 2^{-\gamma_p |l|} \|f\|_p,
\endeq
and 
\beq\label{22}
\Big\|\Big(\sum_{k_0\in \Z}|\sum_{k\ge \frac{k_0}{\alpha}} T_{k_0, \beta\cdot (k-k_0)+l} \pi_{k} \pj_{0}f|^p\Big)^{1/p}\Big\|_p \lesim 2^{-\gamma_p l} \|f\|_p,
\endeq
for some $\gamma_p>0$. Here and throughout the paper $\gamma_{p}$ is a positive constant that is allowed to change from line to line.  We claim that for $p>2$ there exists $\gamma_p>0$ such that
\begin{itemize}
\item if $l\le 0$ and $k> \frac{k_0}{\alpha}-\frac{l}{\b}$, then
\beq\label{lneg}
\|T_{k_0, \beta(k-k_0)+l} \pi_{k} \pj_{0}f\|_p \lesim 2^{-\gamma_p \cdot (\alpha \b k-\b k_0+ l)} \|\pi_{k} \pj_{0}f\|_p,\text{ and}
\endeq
\item if $l>0$ and $k\ge \frac{k_0}{\alpha}$, then
\beq\label{lpoz}
\|T_{k_0, \beta(k-k_0)+l} \pi_{k} \pj_{0}f\|_p \lesim 2^{-\gamma_p \cdot (\alpha \b k-\b k_0+\alpha l)} \|\pi_{k} \pj_{0}f\|_p.
\endeq
\end{itemize}

Before we prove this claim, we will demonstrate how it is used to show \eqref{21} and \eqref{22}. Here we only prove \eqref{21}; the estimate \eqref{22} follows in essentially the same way. 

For a fixed $l\le 0$, we expand the $L^p$ norm on the left hand side of \eqref{21}, and then apply \eqref{lneg} to obtain 
$$\Big(\sum_{k_0\in\Z}\big(\sum_{k> \frac{k_0}{\alpha}-\frac{l}{\b}}2^{-\gamma_p \cdot (\alpha \b k-\b k_0+ l)} \|\pi_{k} \pj_{0}f\|_p\big)^p\Big)^{\frac{1}{p}}.$$
By applying H\"older's inequality to the summation in $k$, we obtain that the above expression is bounded by
$$\Big(\sum_{k_0\in\Z}\sum_{k> \frac{k_0}{\alpha}-\frac{l}{\b}} 2^{-\gamma_p \cdot (\alpha \b k-\b k_0+ l)} \|\pi_{k}\pj_{0} f\|_p^p\Big)^{\frac{1}{p}}.$$
Now we apply Fubini's theorem and exchange the order of summations in $k$ and $k_0$ to further bound this by 

$$2^{\gamma_p \cdot l(\a-1)}\,\big(\sum_{k\in\Z}\|\pi_{k} \pj_{0}f\|_p^p\big)^{\frac{1}{p}}\lesssim 2^{\gamma_p \cdot l(\a-1)}\,\|f\|_p.$$
This finishes the proof of the desired estimate \eqref{21}.\\

It remains to prove \eqref{lneg} and \eqref{lpoz}. The idea is to reduce to the local smoothing estimate provided in Theorem \ref{0813theorem8.1}. We will do this in a slightly more general setting in view of further applications in Section \ref{section-main-theorem}. To this end, given parameters $u>0, w,l\in\Z$, we introduce
\[A_{u,w,l} f(z) := \int f(x-t, y-u [t]^\alpha) \psi_l(2^{\frac{w}\a} (2^{-v} u)^\beta t) \frac{dt}{|t|}, \]
where $v\in\Z$ is such that $u\in [2^v, 2^{v+1})$. Note that $2^{-v}u\in [1,2)$. With this notation we have $T_{k_0,l} = A_{2^{k_0} u^{(0)}_z,0,l}$. The extra parameter $w$ will only be needed in the Section \ref{section-main-theorem}.

\begin{lem}\label{localsmoothinglemma}
Let $r=r_z:\R^2\to [1,2)$ be measurable and $m,k,v,l\in\Z$ such that
\beq M := \max\{ k+l-\frac{w}\a, m+l\alpha-w+v\} \ge 0.\endeq
Then, for every $p>2$ there exists $\gamma_p>0$ such that
\beq \| A_{2^v r_z, w, l} \pi_{k}\pj_{m} f(z)\|_{L^p_z} \lesssim_p 2^{-\gamma_p M} \|f\|_p, \endeq
where the implicit constant depends only on $p$ and $\alpha$.
\end{lem}

Before we proceed to proving this statement let us first note that it directly implies \eqref{lneg} and \eqref{lpoz} (here we use the identity $\beta+1=\alpha\beta$ and $\alpha>1$).

\begin{proof}[Proof of Lemma \ref{localsmoothinglemma}]
We have
\beq\label{eqn:localsmoothinglemmapf5} A_{2^v r_z,w,l} \pi_{k}\pj_{m} f(z) = \int_\R \pi_{k}\pj_{m}f ( x - t, y- 2^v r_z [t]^\alpha) \psi_l ( 2^{\frac{w}\a} r_z^\beta t ) \frac{dt}{|t|}. \endeq
By a change of variables $t\to 2^{l-\frac{w}\a} t$ we obtain
\beq\label{eqn:localsmoothinglemmapf1}
 \int_\R \pi_{k}\pj_{m} f (x-2^{l-\frac{w}\a}t, y-2^{l\a-w+v} r_z [t]^\alpha) \psi_0(r_z^\beta t) \frac{dt}{|t|}. \endeq
\newcommand{\dil}[2]{D_{#1, #2}}
Define $\dil{a}{b} f(x,y) := f(2^a x, 2^b y)$ 
and 
\beq\label{eqn:localsmoothinglemmapf4}
Bf(z) := \int_\R f(x-t, y-r_z  [t]^\alpha) \psi_0(r_z^\beta t) \frac{dt}{|t|}.
\endeq
Changing variables $t\to r^{-\beta}_z t$ and using the identity $\beta=\alpha\beta-1$ we see that $B$ can be written in terms of the averaging operator from \eqref{eqn:avgoperator}:
\beq\label{eqn:localsmoothinglemmapf3} Bf(z) = \int_\R f(x- r_z^{-\beta} t, y - r_z^{-\beta} [t]^\alpha) \psi_0(t) \frac{dt}{|t|} = A_{r_z^{-\beta}} f(z).\endeq
From \eqref{eqn:localsmoothinglemmapf5} -- \eqref{eqn:localsmoothinglemmapf4} we see that
\beq
A_{2^v r_z,w,l} \pi_{k}\pj_{m} f(z)=\dil{-l+\frac{w}\a}{-l\alpha+w-v} B \dil{l-\frac{w}\a}{l\a-w+v} \pi_{k}\pj_{m} f(z).
\endeq 
Since $\dil{a}{b}\pi_{k}\pj_{m} = \pi_{k+a}\pj_{m+b}\dil{a}{b}$, the right hand side in the previous display can be written as
\beq\label{eqn:localsmoothinglemmapf2}
\dil{-l+\frac{w}\a}{-l\alpha+w-v} B \pi_{k+l-\frac{w}\a}\pj_{m+l\a-w+v} \dil{l-\frac{w}\a}{l\a-w+v} f(z).
\endeq
Since conjugation by $\dil{a}{b}$ is an $L^p$ isometry, the $L^p$ norm of \eqref{eqn:localsmoothinglemmapf2} equals
\beq
\|B \pi_{k+l-\frac{w}\a}\pj_{m+l\a-w+v} f\|_p.
\endeq
The claim now follows from Theorem \ref{0813theorem8.1} once we notice that the frequency support of $\pi_{k+l-\frac{w}\a}\pj_{m+l\a-w+v} f$ is contained in the annulus 
$\|(\xi,\eta)\| \sim 2^M$.
\end{proof}

\subsection{The low frequency case}\label{subsection-remainder}

In this subsection, we bound the first term in \eqref{0304e5.12}. We write it as
\beq\label{0825e5.43h}
\begin{split}
 I_{k_0}&=\sum_{k\in \Z}\sum_{l\le \max\{-k, -k_0/\alpha\}} T_{k_0, l} \pi_{k} \pj_{0} f\\
& = \sum_{k\in \Z} \sum_{l\le -\frac{k_0}{\alpha}} T_{k_0, l} \pi_{k} \pj_{0} f+ \sum_{k\leq\frac{k_0}{\alpha}}\sum_{-\frac{k_0}{\alpha}\le l\le -k}T_{k_0, l} \pi_{k} \pj_{0} f \\
& =:I_{k_0}^{(1)}+I_{k_0}^{(2)}\,.
\end{split}
\endeq
The first term can be bounded by the strong maximal function and the maximally truncated Hilbert transform in the $x$-variable. Indeed, comparing $I_{k_0}^{(1)}$ with
$$\sum_{l\le -\frac{k_0}{\alpha}}\int_{\R} \pj_{0} f(x-t, y)\psi_l(t)\frac{dt}{t},$$
we find that their difference is bounded by the strong maximal function. This follows by the same argument as in Section \ref{0814subsection3.1}. 

We pass now to the treatment of the second term, $I_{k_0}^{(2)}$. For this purpose we first define the ``one-dimensional" operator
\beq\label{08311e5.39}
 U_{k_0, l} f(x,y):=\int_{\R} f(x, y-u_z^{(k_0)} [t]^{\alpha})\psi_l([u^{(0)}_z]^{\beta} t)\frac{dt}{t}\,.
\endeq
Notice that when $[t]^{\alpha}$ is an even function this operator is identically zero. We next rewrite the second term as
\beq\label{dec2}
\begin{split}
I_{k_0}^{(2)}=\sum_{l=0}^{\infty}\sum_{k\leq \frac{k_0}{\alpha}-l} (T_{k_0, -\frac{k_0}{\alpha}+l} \pi_{k} \pj_{0}f-U_{k_0, -\frac{k_0}{\alpha}+l} \pi_{k} \pj_{0}f)+\sum_{l=0}^{\infty}\sum_{k\leq\frac{k_0}{\alpha}-l} U_{k_0, -\frac{k_0}{\alpha}+l} \pi_{k} \pj_{0}f.
\end{split}
\endeq
For the contribution coming from the latter term, by the triangle inequality, it suffices to prove that

\beq
\Big\|\sup_{k_0\in \Z} \Big|\sum_{k\le \frac{k_0}{\alpha}-l_0} \int_{\R} (\pi_{k} \pj_{0}f)(x, y-u_z^{(k_0)} [t]^{\alpha})\psi_{-\frac{k_0}{\alpha}+l_0}([u^{(0)}_z]^{\beta} t)\frac{dt}{t}\Big|\Big\|_p \lesim 2^{-\gamma_p \cdot l_0}\|f\|_p,
\endeq
for each $l_0\in \N$, and for a constant $\gamma_p$ depending only on $p$. This further follows from the pointwise bound
\beq\label{ms}
\Big|\sum_{k\le \frac{k_0}{\alpha}-l_0} \int_{\R} (\pi_{k} \pj_{0} f)(x, y-u_z^{(k_0)} [t]^{\alpha})\psi_{-\frac{k_0}{\alpha}+l_0}([u^{(0)}_z]^{\beta} t)\frac{dt}{t}\Big| \lesim 2^{-\gamma \cdot l_0} M_S f(x, y),
\endeq
which is again a consequence of the mean zero property of $\psi_{-\frac{k_0}{\alpha}+l_0}(t)\cdot \frac{1}{t}$. Indeed, \eqref{ms} follows from classical Calder\'{o}n-Zygmund theory. We leave the details for the interested reader.\\

Turning our attention towards the contribution from the former term on the right hand side of \eqref{dec2}, we notice that it is enough to show that
\beq
\Big\| \sup_{k_0\in \Z}\Big|\sum_{k\le \frac{k_0}{\alpha}-l_0}\big(T_{k_0, -\frac{k_0}{\alpha}+l_0}  \pi_{k} \pj_{0} f-U_{k_0, -\frac{k_0}{\alpha}+l} \pi_{k} \pj_{0} f\big)\Big|\Big\|_p \lesim 2^{-\gamma_p \cdot l_0} \|f\|_p,
\endeq
for each $l_0\in \N$. This in turn follows from 
\begin{claim}
In the above setting, for each $l_0, j_0\in \N$ and $k_0\in\Z$ the following estimate holds uniformly
\beq\label{s21}
\Big\| \sup_{k_0\in \Z}\Big|T_{k_0, -\frac{k_0}{\alpha}+l_0}  \pi_{\frac{k_0}{\alpha}-l_0-j_0} \pj_{0} f-U_{k_0, -\frac{k_0}{\alpha}+l} \pi_{\frac{k_0}{\alpha}-l_0-j_0} \pj_{0} f\Big|\Big\|_p \lesim 2^{-\gamma_p \cdot \max\{l_0, \frac{j_0}{9}\}} \|f\|_p\,.
\endeq
\end{claim}

\begin{proof}
We first look at the case $j_0\le 9\cdot l_0$. Based on the treatment of $I_{k_0}^{(2)}$, it is enough to show that
\beq
\|T_{k_0, -\frac{k_0}{\alpha}+l_0}  \pi_{\frac{k_0}{\alpha}-l_0-j_0} \pj_{0} f\|_p \lesim 2^{-\gamma_p \cdot l_0}\|f\|_p,
\endeq
for each fixed $k_0, l_0$ and $j_0$. This follows from Lemma \ref{localsmoothinglemma}.

Next we consider the case $j_0\ge 9\cdot l_0$. It suffices to prove the pointwise bound
\beq\label{dif}
\big|T_{k_0, -\frac{k_0}{\alpha}+l_0}  \pi_{\frac{k_0}{\alpha}-l_0-j_0} \pj_{0} f- U_{k_0, -\frac{k_0}{\alpha}+l} \pi_{\frac{k_0}{\alpha}-l_0-j_0}  \pj_{0}f\big| \lesim 2^{-\gamma\cdot j_0} \cdot M_S (f),
\endeq
for some positive $\gamma>0$. By scaling it suffices to look at the case $k_0=0$. We now observe that the left hand side of \eqref{dif} can be written as
\beq\label{evencancellation}
\Big|\int_{\R} \Big[ (\pi_{-l_0-j_0} \pj_{0} f)(x-t, y-u_z^{(0)}[t]^{\alpha})-(\pi_{-l_0-j_0} \pj_{0} f)(x, y-u_z^{(0)}[t]^{\alpha})\Big] \psi_{l_0}([u^{(0)}_z]^{\beta}t)\frac{dt}{t}\Big|.
\endeq
Now our claim follows by applying the fundamental theorem of calculus in the first variable of $\pi_{-l_0-j_0} \pj_{0}f$.
\end{proof}

\section{Proof of Theorem \ref{thm:mainmax}}\label{section-main-theorem}

The proof is organized as follows:
\begin{itemize}
\item In Section \ref{prel} we reduce our proof to the exponential decay estimate \eqref{2104e1.6}.

\item In Section \ref{sec:maxthmpgtr2} we prove this decay estimate for $p>2$, by only assuming $u$ to be a measurable function. This recovers the result of Marletta and Ricci \cite{MR}.

\item In Section \ref{section-psmall} we prove the estimate \eqref{2104e1.6} for $p\le 2$. This relies on the Lipschitz assumption of the function $u$, a condition that is used when applying a suitable change of variables (see \eqref{Lipschitzchangevariables}). To enable that change of variables, we will introduce several auxiliary functions and make use of Lemma \ref{2704lemma3.6} on the boundedness of maximal operators along curves in lacunary directions. The proof of this lemma is provided in Section \ref{section-lacunary}.
\end{itemize}

\subsection{Preliminaries}\label{prel}

Since we are dealing with a positive operator we may assume without loss of generality that $f\ge 0$. Futhermore, we may assume that $u(x,y)>0$  for all $(x,y)\in\R^2$. We will adopt the notation \eqref{08300e5.4} --  \eqref{08300e5.6} from the previous section.

For a positive real number $u$ and $l\in\Z$ we define
\[ A_{u,l} f(z) = \int_\R f(x-t, y-u [t]^\alpha) \psi_l(2^{\frac{v}\a} (2^{-v} u)^{\beta} t) \frac{dt}{|t|},  \]
where $v\in\Z$ is such that $u\in [2^v, 2^{v+1})$ and $\beta := \frac1{\alpha-1}$. In particular, $2^{-v} u \in [1,2)$.
Observe that
\[ \mathcal{M}^{(\alpha)}_{u,\varepsilon_0} f(z) \lesssim \sup_{\substack{l\,:\,v_z\in E_{l}}} A_{u_z,l} f(z),\]
where
\beq\label{eqn:defEl}
E_l := \{ v\in\Z\,:\,2^{l}\le c_\alpha \varepsilon_0 2^{\frac{v}\a}\},
\endeq
for $v\in\Z$ and $c_\alpha$ is a fixed constant depending only on $\alpha$. Linearizing the supremum we introduce the operator
\[ \mathcal{M} f(z) := A_{u_z, l_z} f(z), \]
where $z\mapsto l_z$ is an arbitrary measurable map $\R^2\to \Z$ such that $v_z\in E_{l_z}$ for all $z$. To prove our theorem it suffices to show that
\[ \|\mc{M}f\|_p \lesssim_{\alpha,p} \|f\|_p\]
for all $1<p<\infty$.
By the Fourier inversion formula we have
\[ A_{u_z,l} f(z) = \int_{\R^2} \widehat{f}(\xi,\eta) e^{ix\xi+iy\eta} m_{l}(z,\xi,\eta) d(\xi,\eta), \]
with the symbol $m_{l}(z,\xi,\eta)$ given by
\beq\label{eqn:Aulsymbol} m_{l}(z,\xi,\eta) := \int_\R e^{-i(\xi t+\eta u_z [t]^\alpha)} \psi_l(2^{\frac{v_z}\a} (u^{(0)}_z)^\beta t)\frac{dt}{|t|}, \endeq
where $u^{(0)}_z := 2^{-v_z} u_z$, as defined in \eqref{08300e5.5}.
 As in our previous analysis, our approach relies on decomposing our operator relative to the behavior of the phase function of the multiplier. Our initial focus will be on the behavior of the $\eta$ component of the phase which corresponds in the spatial variable to the component containing the vector field $u$. Thus, we will discuss the following two cases:
 \begin{itemize}
 \item the $\eta-$low frequency regime: $|\eta| 2^{\a l}\leq 1$;

 \item the $\eta-$high frequency regime: $|\eta|2^{\a l}> 1$.
 \end{itemize}

Following the above description, we perform a Littlewood-Paley decomposition in the $y$-variable and write
\beq
A_{u_z,l_z} = \sum_{k\in\Z} A_{u_z,l_z} \pj_{k} = \sum_{k\le -\alpha l_z} A_{u_z,l_z} \pj_{k} + \sum_{k>-\alpha l_z} A_{u_z,l_z}\pj_{k}.
\endeq

\noindent \underline{Case 1.} $\eta-$low frequency.\\

The sum over $k\le -\alpha l_z$ can be estimated by the strong maximal operator and is therefore bounded on all $L^p$, $1<p\le \infty$. This follows by the same arguments as in Section \ref{0814subsection3.1}.\\

\noindent \underline{Case 2.} $\eta-$high frequency.\\

The remaining term equals\footnote{Here and in the following our notation will ignore the issue that $-\alpha l_z$ might not be an integer as this can be easily addressed by setting $P^{(i)}_\alpha=P^{(i)}_{\lfloor\alpha\rfloor}$.}
\beq\label{eqn:maxthmpf1} \sum_{k\in\N} A_{u_z,l_z}\pj_{-\alpha l_z+k}=\sum_{k\in\N}  \sum_{m\in\Z} A_{u_z,l_z} \pi_{m}\pj_{-\alpha l_z+k}. \endeq

Fix a $k\in \N$, and for the following heuristic, let us imagine for a moment that $u_z$ is a constant. Then from \eqref{eqn:Aulsymbol}, the symbol of $A_{u_z,l_z}\pj_{-\alpha l_z+k}$ becomes
$$\int_{\R} e^{i(2^{l_z} u_z^{-\frac{1}{\alpha}}\xi t+\eta 2^{\alpha \cdot l_z} [t]^{\alpha})}\psi_0(t)\frac{dt}{|t|}.$$
From the stationary phase principle it is plausible to expect an exponential decay for the $L^p$ norm of the term \eqref{eqn:maxthmpf1} in terms of $k\in \N$. Indeed, this is the case for Hilbert transforms along one-variable curves (see the estimate \eqref{1702e4.15}).  However, in the present situation we do not know how to exhibit such an exponential decay. To remedy this we first remove from the term \eqref{eqn:maxthmpf1} the $\xi-$low frequency component; the remaining part then admits an exponential decay estimate in $k\in \N$. \\

We split the term into \eqref{eqn:maxthmpf1} in two components corresponding to
\begin{itemize}
\item the $\xi-$low frequency regime: $|2^{l_z} u_z^{-\frac{1}{\alpha}}\xi|<1$;

\item the $\xi-$high frequency regime: $|2^{l_z} u_z^{-\frac{1}{\alpha}}\xi|\geq 1$.
\end{itemize}
That is, we write \eqref{eqn:maxthmpf1} as
\beq\label{xidec}
\sum_{k\in\N} A_{u_z,l_z} \sum_{m\le -l_z+\frac{v_z}\alpha} \pi_{m}\pj_{-\alpha l_z+k} + \sum_{k\in\N} \sum_{m\in\N}  A_{u_z,l_z} \pi_{-l_z+\frac{v_z}\alpha+m}\pj_{-\alpha l_z+k}.
\endeq

\noindent \underline{Case 2.1.} $\xi-$low frequency.\\

This corresponds to the first term in \eqref{xidec}. The contribution from this term can be controlled by the strong maximal function using the same argument as in the $\eta$-low frequency case. We omit the details. 
\\
\\
\noindent \underline{Case 2.2.} $\xi-$high frequency.\\

To handle the latter term on the right hand side of \eqref{xidec}, it suffices to show that for every $p>1$ there exists $\gamma_p>0$ such that
\beq\label{2104e1.6}
\Big\| \sum_{m\in\N}  A_{u_z, l_z}\pi_{-l_z+\frac{v_z}\alpha+m}\pj_{-\alpha l_z+k} f \Big\|_p \lesim 2^{-\gamma_p k}\|f\|_p.
\endeq

In view of the Littlewood-Paley square function estimate,  the argument splits naturally into two cases: $p>2$ and $p\le 2$.

\subsection{The case $p>2$}\label{sec:maxthmpgtr2}
In this section we reprove the result of Marletta and Ricci \cite{MR} and, in particular, we only require the function $u$ to be measurable. We also do not need to consider a truncated maximal function; under the measurability assumption, the truncated case is equivalent to the untruncated case, by a scaling argument. Hence we take $\varepsilon_0=\infty$.

We first remark that the left hand side of \eqref{2104e1.6} is bounded from above by
\beq\label{2106e3.10}
\Big\|\Big( \sum_{l\in \Z}\Big| \sum_{m\in\N}A_{u_z, l} \pi_{-l+\frac{v_z}\alpha+m}\pj_{-\alpha l+k} f \Big|^p \Big)^{1/p}\Big\|_p\,.
\endeq
Hence, it suffices to prove that for every $l\in\Z$ one has
\beq\label{2106e3.10h}
\Big\|\sum_{m\in\N}A_{u_z, l} \pi_{-l+\frac{v_z}\alpha+m}\pj_{-\alpha l+k} f\Big\|_p\lesssim 2^{-\gamma_p k} \|f\|_p\:.
\endeq

For a fixed $v\in\Z$, recall the definition of $u_z^{(v)}$ from \eqref{08300e5.6}.
Then the left hand side of \eqref{2106e3.10h} can be bounded by
\beq\label{2204e1.17}
\Big\| \Big(\sum_{v\in \Z}\Big|\sum_{m\in\N}A_{u_z^{(v)},l}\pi_{-l+\frac{v}\alpha+m}\pj_{-\alpha l+k} f\Big|^p\Big)^{1/p} \Big\|_p.
\endeq
Taking in account \eqref{2106e3.10} --\eqref{2204e1.17}, we see that it suffices to show that for every $p>2$ there exists $\gamma_p>0$ such that for every $k,m\in\N$ and $l,v\in\Z$ we have

\beq\label{contrdec}
\Big\|A_{u_z^{(v)},l}\pi_{-l+\frac{v}\alpha+m}\pj_{-\alpha l+k} f(z)\Big\|_p \lesim 2^{-\gamma_p\,\max\{k,\,m\}} \,\|f\|_p\,.
\endeq
Noting that $A_{u_z^{(v)},l} = A_{2^v u_z^{(0)},v,l}$ we recognize this as precisely the conclusion of Lemma \ref{localsmoothinglemma}.

\subsection{$L^p$ estimates for $p\le 2$}\label{section-psmall}
In this subsection we will prove \eqref{2104e1.6} for all $1<p\le 2$. Note that we are working here with two extra assumptions, as compared with the previous subsection:
\begin{itemize}
\item First, we assume that $u$ is Lipschitz. This is because if $u$ is only assumed to be measurable, then it is possible for $\mc{M}_{u,\epsilon_0}^{(\alpha)}$ to be unbounded for each $p\le 2$. One may verify this by simply plugging the characteristic function of the unit ball into the inequality.
\item Second, the truncation $\epsilon\le \epsilon_0$ will play a crucial role. That is, in effect we will have to take into account the range restriction of $l_z$ expressed by \eqref{eqn:defEl}.
\end{itemize}
Both of these assumptions will only be used to ensure the validity of the Lipschitz change of variables in \eqref{Lipschitzchangevariables}.

Let us begin with the proof. By interpolation with the case $p>2$, it suffices to prove an estimate without decay; that is, it suffices to show that
\beq\label{2204e1.23}
\Big\| \sum_{m> -l_z +\frac{v_z}{\alpha}}  A_{u_z, l_z}\pi_{m}\pj_{-\alpha l_z+k} f(z) \Big\|_{L^p_z} \lesim \|f\|_p.
\endeq

Again we get rid of the $z$-dependence of $l_z$ by inserting a sum over all possible values of $l_z$, similar to \eqref{2106e3.10}. In particular, we bound the expression inside the $L^p$ norm on the left hand side of \eqref{2204e1.23} by\footnote{At this point we fix $\epsilon_0$ possibly depending on $\alpha$ so that this estimate becomes valid.}
\beq\label{2106e3.18}
\Big( \sum_{l\,:\,v_z\in E_{l}}\Big| \sum_{m> -l +\frac{v_z}{\alpha}}  A_{u_z, l}   \pi_{m}\pj_{-\alpha l+k} f(z) \Big|^2 \Big)^{1/2}.
\endeq
We view the above expression as an $l^2$-valued function. When estimating the $L^p$ norm of it, we would like to appeal to certain interpolation for an $l^q$-valued function. However, for such a purpose, it is not convenient to carry the summation $\sum_{l: v_z\in E_l}$. It will be more convenient to see \eqref{2106e3.18} as
\beq\label{0814e7.15h}
\Big( \sum_{l\in \Z}\Big| \mathbbm{1}_{E_{l}}(v_z) \sum_{m> -l +\frac{v_z}{\alpha}}  A_{u_z, l} \pi_{m}\pj_{-\alpha l+k} f(z) \Big|^2 \Big)^{1/2}.
\endeq
To estimate \eqref{0814e7.15h} we will run an interpolation argument which requires one to consider estimates of the form
\beq\label{0814e7.16h}
\Big\|\Big( \sum_{l\in \Z}\Big| \mathbbm{1}_{E_{l}}(v_z) \sum_{m> -l +\frac{v_z}{\alpha}} A_{u_z, l}\pi_{m}\pj_{-\alpha l+k} f(z) \Big|^q \Big)^{1/q}\Big\|_p\lesim \Big\|\Big( \sum_{l\in \Z}\Big|\pj_{-\alpha l+k} f\Big|^q \Big)^{1/q}\Big\|_p. 
\endeq
When $q=\infty$, it is not difficult to see that the expression inside the $L^p$ norm on the left hand side of the last display is bounded by 
\beq\label{0814e7.17h}
\mc{M}^{(\alpha)}_{u, \epsilon_0} (M_S f)(z).
\endeq
Here $M_S$ denotes the strong maximal operator. Recall that we already proved the $L^p$ boundedness of $\mc{M}^{(\alpha)}_{u, \epsilon_0}$ for all $p>2$ in Section \ref{sec:maxthmpgtr2}. This implies
\beq\label{0814e7.18h}
\eqref{0814e7.16h} \text{ holds for all } p>2 \text{ and } q=\infty.
\endeq
We will interpolate estimate \eqref{0814e7.18h} with
\beq\label{0504e3.10}
\Big\|\mathbbm{1}_{E_{l}}(v_z) \sum_{m>-l+\frac{v_z}{\alpha}}A_{u_z, l}\pi_{m}\pj_{-\alpha l+k} f(z) \Big\|_p\lesim \|f\|_p,
\endeq
for all $p>1$. 
Suppose for the moment that \eqref{0504e3.10} holds. This gives
\beq\label{1030e6.17}
\eqref{0814e7.16h}  \text{ holds for all } p=q>1.
\endeq
Then, interpolating \eqref{1030e6.17} with \eqref{0814e7.18h} we obtain that
\beq\label{08300e6.19}
\|\eqref{0814e7.15h}\|_p \lesim \|f\|_p,
\endeq
for all $p>\frac{4}{3}$. Combining this with \eqref{0814e7.17h} implies
\beq\label{0814e7.21h}
\eqref{0814e7.16h} \text{ holds for }  p>\frac{4}{3}\text{ and }q=\infty.
\endeq
We interpolate \eqref{0814e7.21h} and \eqref{0504e3.10} to see that \eqref{08300e6.19} holds for all $p>\frac{8}{7}$. Repeating this interpolation procedure sufficiently many times, we obtain that \eqref{0814e7.15h} is bounded on $L^p$ for all $p>1$. We learnt this interpolation trick from Nagel, Stein and Wainger \cite{NSW}.\\

Before we turn to the proof of \eqref{0504e3.10} we need to introduce some new notation. For $n\in\Z$ and $z\in\R^2$ we define
\[ u_n (z) := \left\{ \begin{array}{ll} u_z &
\text{if}\;v_z = n,\\
2^n & \text{if}\;v_z \not=n.\end{array}\right. \]
Note carefully that $u_{n}(z)$ is different from $u_z^{(n)}$. Both take values in the interval $[2^{n}, 2^{n+1})$. 
However, the function $u_{n}$ retains more of the regularity of $u$ than $z\mapsto u_z^{(n)}$. This will be important during the proof of \eqref{0504e3.10} (see \eqref{0814e7.35h}, \eqref{0830e6.39h}).

Eliminating the $z$-dependence of $v_z$ by introducing another sum we estimate the left hand side of \eqref{0504e3.10} by
\beq\label{2204e1.31}
\Big\| \Big(\sum_{v\in E_l} \Big| \sum_{m>-l+\frac{v}{\alpha}}A_{u_v(z),l}\pi_{m}\pj_{-\alpha l+k} f(z) \Big|^2\Big)^{1/2} \Big\|_p.
\endeq
By the triangle inequality this is no greater than
\beq\label{2204e1.31b}
\sum_{m>0} \Big\| \Big(\sum_{v\in E_l} \Big| A_{u_{v}(z),l}\pi_{m-l+\frac{v}{\alpha}}\pj_{-\alpha l+k} f(z) \Big|^2\Big)^{1/2} \Big\|_p.
\endeq
Thus \eqref{0504e3.10} will follow if we can show that there exist constants $\gamma_p>0$ such that for every $k,m\in\N, l\in\Z$ we have
\beq\label{2204e1.34}
\Big\| \Big(\sum_{v\in E_l} \Big| A_{u_{v}(z),l}\pi_{m-l+\frac{v}{\alpha}}\pj_{-\alpha l+k} f(z) \Big|^2\Big)^{1/2} \Big\|_p\lesim 2^{-\gamma_p \max\{k,\,m\}}\,\|f\|_p.
\endeq
We first prove \eqref{2204e1.34} for $p=2$. 
By Lemma \ref{localsmoothinglemma} we obtain that for all $p>2$ there exists $\gamma_p>0$ such that
\beq
\Big\| A_{u_{v}(z),l}\pi_{m-l+\frac{v}{\alpha}}\pj_{-\alpha l+k} f(z)\|_p\lesim  2^{-\gamma_p \max\{k,\,m\}}\,\|f\|_p.
\endeq
Thus, in order to show \eqref{2204e1.34} for $p=2$ it will be enough to show the estimate
\beq
\Big\| A_{u_{v}(z),l}\pi_{m-l+\frac{v}{\alpha}}\pj_{-\alpha l+k} f(z)\|_p\lesssim \|f\|_p.
\endeq
for an arbitrary $p<2$. In the following we will prove something stronger, namely that
\beq\label{eqn:627}
\Big\| A_{u_{v}(z),l} f(z)\|_p\lesssim \|f\|_p
\endeq
holds for all $p>1$ provided that $v\in E_l$.
Before commencing the proof of \eqref{eqn:627}, we need to introduce another auxiliary function. If at a point $z=(x ,y)$ we have $v_z=n$, then let $\tilde{u}_{n}(z) := u_z$. We now extend $\tilde{u}_{n}$ to the whole space, requiring only that
\beq
\tilde{u}_{n}(z)\in [2^{n}, 2^{n+1}) \text{ and } \|\tilde{u}_{n}\|_{\mathrm{Lip}}\le 2 \|u\|_{\mathrm{Lip}}.
\endeq
The advantage is that, unlike $u_n$, the function $\tilde{u}_n$ is Lipschitz continuous with controlled Lipschitz norm. This will later be of key importance in ensuring the validity of the Lipschitz change of variables \eqref{Lipschitzchangevariables}. Most importantly, the way that $\tilde{u}_{n}$ and $u_{n}$ are designed is such that they stay very ``close'' to each other, in the sense that if we define the following maximal operator:
\beq\label{dyadicmonomial}
\mc{M}_{\mathrm{dyad}} f(z):=\sup_{j\in \Z} \sup_{\epsilon>0}\frac{1}{2\epsilon} \int_{-\epsilon}^{\epsilon} |f(x-t, y-2^j [t]^{\alpha})|dt.
\endeq
then we have the pointwise estimate
\beq\label{0814e7.35h}
A_{u_v(z),l} f(z)\lesssim \mc{M}_{\mathrm{dyad}} f(z)+A_{\tilde{u}_v(z),l} f(z)
\endeq
for all $v,l\in\Z, z\in\R^2$. The maximal operator $\mc{M}_{\mathrm{dyad}}$ is bounded in $L^p$ for all $p>1$.
\begin{lem}\label{2704lemma3.6}
For each $p>1$, we have
\beq\label{Mcont}
\|\mc{M}_{\mathrm{dyad}}f\|_p \lesim \|f\|_p.
\endeq
The implicit constant depends only on $p$.
\end{lem}
We postpone the proof of this fact until the end of this subsection and turn our attention toward the second term in \eqref{0814e7.35h}. By Minkowski's integral inequality, the $L^p$ norm of $A_{\tilde{u}_v(z),l} f(z)$ is at most
\[ \int_\R \left( \int_{\R^2} (f(x-t, y- \tilde{u}_v(z) [t]^\alpha) \psi_l( 2^{\frac{v}\a} r_z t))^p dz \right)^{1/p} \frac{dt}{|t|}.\]
where $r_z := (2^{-v} \tilde{u}_n(z) )^\beta\approx 1$. To bound this term, we apply the change of variables
\beq\label{Lipschitzchangevariables}
\left\{
\begin{array}{ll}
x_1 &:= x-t\\
y_1 &:= y-\tilde{u}_{v}(x, y)\,[t]^{\alpha}
\end{array}\right. .
\endeq
This is the only point at which we use the Lipschitz regularity of $u$ and it is also where we need to exploit the range restriction on $l$, i.e. that $2^l\le c_\alpha \epsilon_0 2^{\frac{v}\a}$.
Indeed, computing the determinant of the corresponding Jacobian $J$, we have
\beq\label{jacob}
\textrm{det}\,J=\left| \frac{\partial(x_1,y_1)}{\partial(x,y)}\right| = \left| \begin{array}{cc}
1 & 0\\
-\frac{\partial \tilde{u}_{v}}{\partial x}\,[t]^{\alpha} & 1- \frac{\partial \tilde{u}_{v}}{\partial y}\,[t]^{\alpha}
\end{array} \right|=1- \frac{\partial \tilde{u}_{v}}{\partial y}\,[t]^{\alpha}\,.
\endeq
Observe that $t$ in \eqref{Lipschitzchangevariables} obeys
$$|[t]^{\alpha}|\approx 2^{l-\frac{v}{\a}}\le c_\alpha \epsilon_0$$
and hence
\beq\label{detj}
|\textrm{det}\,J|= \Big|1- \frac{\partial \tilde{u}_{n}}{\partial y}\,[t]^{\alpha}\Big|\geq 1- \tilde{c}_\alpha\|\tilde{u}_{n}\|_{\mathrm{Lip}} \epsilon_0 \ge \frac12,
\endeq
where we have chosen $\epsilon_0$ to be smaller than $(2 \tilde{c}_\alpha \|u\|_{\mathrm{Lip}})^{-1}$. Here $\tilde{c}_\alpha$ is a constant depending only on $\alpha$. This shows that the change of variables \eqref{Lipschitzchangevariables} is valid and
therefore
\beq
\|A_{\tilde{u}_v(z),l} f \|_p \lesssim \|f\|_p
\endeq
holds for all $p>1$. Hence we can infer that also \eqref{eqn:627} holds. This concludes the proof of \eqref{2204e1.34} in the case $p=2$.\\

It remains to prove \eqref{2204e1.34} for the remaining values of $p$. By interpolation with \eqref{2204e1.34} at $p=2$, it suffices to prove that we have the estimate without decay,
\beq\label{eqn:635}
\Big\| \Big(\sum_{v\in E_l} \Big| A_{u_{v}(z),l}\pi_{m-l+\frac{v}{\alpha}} \pj_{-\alpha l+k}f(z) \Big|^2\Big)^{1/2} \Big\|_p\lesssim \|\pj_{-\alpha l+k}f\|_p\sim \Big\|\Big(\sum_{v\in \Z}\Big|\pi_{-l+m+\frac{v}{\alpha}}\pj_{-\alpha l+k}f\Big|^2\Big)^{1/2}\Big\|_p
\endeq
for all $p>1$. We will again use the interpolation trick from Nagel, Stein and Waigner \cite{NSW}. Thus we consider the more general estimate
\beq\label{2204e1.48}
\Big\| \Big(\sum_{v\in E_l} \Big| A_{u_{v}(z),l}\pi_{m-l+\frac{v}{\alpha}} \pj_{-\alpha l+k}f(z) \Big|^q\Big)^{1/q} \Big\|_p\lesssim \Big\|\Big(\sum_{v\in \Z}\Big|\pi_{-l+m+\frac{v}{\alpha}}\pj_{-\alpha l+k}f\Big|^q\Big)^{1/q}\Big\|_p
\endeq
for $1<p<\infty$ and $1<q\leq \infty$.
 
Recall that in Section \ref{sec:maxthmpgtr2} we proved that
\beq\label{0890e6.39}
\|\mc{M}^{(\alpha)}_{u, \epsilon_0}  f\|_p \lesim \|f\|_p\text{ for all } p>2.
\endeq
Moreover, we have the pointwise bound
\beq\label{0830e6.39h}
|A_{u_v(z),l}f(z)|\lesim \mc{M}_{\mathrm{dyad}}f(z)+\mc{M}^{(\alpha)}_{u, \epsilon_0}  f(z)\text{ for all }v,\,l\in \Z.
\endeq

Note that this pointwise estimate does not hold if $u_v$ is replaced by $\tilde{u}_{v}$, because we do not know how $\tilde{u}_{v}$ behaves outside of the region where it coincides with the original Lipschitz function $u$.

From \eqref{0890e6.39}, \eqref{0830e6.39h} and Lemma \ref{2704lemma3.6} we see that \eqref{2204e1.48} holds for $q=\infty$ and $p>2$. Moreover, by \eqref{eqn:627} we know that \eqref{2204e1.48} holds for all $q=p>1$. By interpolation we obtain \eqref{2204e1.48} for $q=2$ and $p>4/3$. This in turn implies that
\beq
\|\mc{M}^{(\alpha)}_{u, \epsilon_0}  f\|_p \lesim \|f\|_p,
\endeq
for all $p>4/3$. Iterating this process sufficiently many times we prove \eqref{eqn:635} for every $p>1$. This finishes the proof of \eqref{2204e1.34} and thereby also concludes the proof of Theorem \ref{thm:mainmax}.\\

We close this subsection with several remarks regarding the last part of the proof:
\begin{enumerate}
\item In the proof of \eqref{eqn:627}, we did not need the full strength of Lemma \ref{2704lemma3.6}. For that purpose it would have been sufficient to a fixed $j$ in the definition of $\mc{M}_{\mathrm{dyad}}$ instead of taking the supremum over $j\in\Z$. The full strength of Lemma \ref{2704lemma3.6} is only needed during the interpolation procedure that is used for proving \eqref{2204e1.34} for $p<2$.

\item One might wonder why we did not use the auxiliary function $\tilde{u}_{v}$ right away, rather than first introducing the auxiliary function $u_{v}$. Again, the point is that this would cause the interpolation argument for $p<2$ to fail, because the pointwise estimate \eqref{0830e6.39h} would no longer be available.
\end{enumerate}

\subsection{Proof of Lemma \ref{2704lemma3.6}}\label{section-lacunary}

Before we start the proof of this lemma, we need to emphasize that this lemma was first established by Hong, Kim and Yang \cite{HKY}. Indeed, their results have a much wider scope, in the sense that they considered general polynomials in all dimensions. Due to such generality, their proof is significantly more intricate. For the sake of completeness, we provide here the proof of Lemma \ref{2704lemma3.6}. 

We will present the argument for $\alpha=2$ and $[t]^2=t^2$; the general case follows \emph{mutatis mutandis}.\\

For $k,\,j\in \Z$ we set
\beq\label{modjk}
M^j_k f(x, y):=\Big|\frac{1}{2^k}\int_{2^k}^{2^{k+1}}f(x-t, y-2^{-j} t^2)dt\Big|.
\endeq
Fix $j\in \Z$ consider the maximal operator along the parabola $(t, 2^{-j} t^2)$, which is given by
\beq
M^j f(x, y):=\sup_{k\in \Z} M^j_k f(x, y).
\endeq
Hence
\beq\label{Mmdef}
\mc{M}_{\mathrm{dyad}}f(x, y)=\sup_j M^j f(x, y)=\sup_{k,j}M^j_kf(x, y).
\endeq
A key idea to prove the $L^2$ bounds for $\mc{M}_{\mathrm{dyad}}$ is to compare its components $M^j_k$ with some suitable smoother versions $\widetilde{M}^j_k$ and prove bounds for the corresponding square-function
\beq\label{1912ee1.7}
\Big(\sum_{j\in \Z}\sum_{k\in \Z} |M^j_k - \widetilde{M}^j_k|^2 \Big)^{1/2}\,.
\endeq
As a first attempt for finding a good candidate for $\widetilde{M}^j_k$ one may consider the \textit{linearized} model
\beq\label{1401ee2.5}
\bar{M}^j_kf(x, y):=\frac{1}{2^k\cdot 3\cdot 2^{2k}} \int_{2^k}^{2^{k+1}}\int_{2^{2k}}^{2^{2k+2}}f(x-t, y-2^{-j}\tau)dtd\tau.
\endeq
It turns out that while the operator $\big(\sum_{k\in \Z} |M^j_k - \bar{M}^j_k|^2 \big)^{1/2} $
is indeed bounded on $L^2$, one does not have a good control over the double indexed sum in \eqref{1912ee1.7}. For this reason, one needs to choose a variant that is closer in spirit to \eqref{modjk}, which is simultaneously smoother and preserves the \textit{quadratic} nature of the initial operator:
\beq\label{1401ee2.15}
\widetilde{M}^j_k f(x, y):= \frac{1}{2^k \cdot 2^k}\int_{2^k}^{2^{k+1}}\int_{2^k}^{2^{k+1}}f(x-t, y-2^{-j}\tau^2)dt d\tau.
\endeq
It is not difficult to see that we have the pointwise bound
\beq
\widetilde{M}^j_k f(x, y)\lesim M_S f(x, y)\,,
\endeq
and thus, in order to prove the $L^p$ bound of $\mc{M}_{\mathrm{dyad}}$, it suffices to prove the $L^p$ bound for
\beq\label{0814e7.52h}
\Big(\sum_{j\in \Z}\sum_{k\in \Z} |M^j_k - \widetilde{M}^j_k|^2 \Big)^{1/2} .
\endeq
As opposed to the general $L^p$ case, the proof for the $L^2$ boundedness of \eqref{0814e7.52h} is significantly simpler and thus we choose to present it first.

\subsubsection{\bf $L^2$ boundedness}

In order to prove the $L^2$ bounds we rely on Plancherel's theorem. Naturally, we start by analyzing the multipliers for the corresponding operators $M^j_k$ and $\widetilde{M}^j_k$.

The multiplier of $M^j_k$ is given by
\beq\label{0814e7.54h}
m^j_k(\xi, \eta):=\int_1^2 e^{i2^k t\xi+i 2^{2k-j}t^2\eta}dt\,,
\endeq
while the multiplier of the operator \eqref{1401ee2.15} is given by
\beq\label{0814e7.53h}
\begin{split}
\widetilde{m}^j_k(\xi, \eta):=\int_1^{2}\int_1^{2} e^{i 2^k t \xi+ i2^{2k-j}\tau^2\eta}dtd\tau\,.
\end{split}
\endeq
By Plancherel's theorem, it suffices to prove that
\beq
\sum_{j\in \Z, k\in \Z}\left| m^j_k(\xi,\eta)-\widetilde{m}^j_k(\xi, \eta)\right|^2 \le C,
\endeq
for all $(\xi, \eta)\in \R^2$ and some universal constant $C>0$. This in turn follows from
\beq\label{0814e7.56h}
\sum_{j\in \Z, k\in \Z}\left| m^j_k(\xi,\eta)-\widetilde{m}^j_k(\xi, \eta)\right| \lesim 1, \text{ for all } \xi\sim1 \text{ and } \eta\sim 1.
\endeq

\begin{itemize}
\item \textbf{Case I: $k+10\ge j\ge 0$.}  Applying van der Corput's lemma, we have
\beq
\left| m^j_k(\xi,\eta)-\widetilde{m}^j_k(\xi, \eta)\right| \lesim 2^{-(k-\frac{j}{2})}.
\endeq
The last display is easily seen to be summable within the range $k\ge j\ge 0$.\\

\item\textbf{Case II: $4k\ge j> k+10\ge 0$ or $k\ge 0, j\le 0$ or $j\le 4k\le 0$.} For simplicity, we only detail the case $4k\ge j> k\ge 0$. Under this assumption, the phase functions in \eqref{0814e7.53h} and \eqref{0814e7.54h} do not admit any critical point. Thus,
\beq
\left| m^j_k(\xi,\eta)-\widetilde{m}^j_k(\xi, \eta)\right| \lesim 2^{-k}.
\endeq
By summing over $10+k<j\le 4k$, we obtain the upper bound $3k\cdot 2^{-k}$, which is summable in $k\in \N$.\\

\item\textbf{Case III: $j\ge 4k\ge 0$ or $k\le 0, j\ge 0$ or $k\le j\le 0$.} Again we only detail one case, that of $j\ge 4k\ge 0$. By definition, we have
\beq
\begin{split}
& m^j_k(\xi,\eta)-\widetilde{m}^j_k(\xi, \eta)=\int_1^{2}e^{i 2^k t\xi+i 2^{2k-j}t^2\eta}dt- \int_1^{2}\int_1^2 e^{i 2^k t\xi+i 2^{2k-j}\tau^2\eta}dt d\tau\\
&= \int_1^{2}e^{i2^k t\xi}\left(e^{i 2^{2k-j}t^2\eta}-1\right)dt-\int_1^{2}\int_1^2 e^{i2^k t\xi}\left(e^{i 2^{2k-j}\tau^2\eta}-1\right)dt d\tau.
\end{split}
\endeq
By the fundamental theorem of calculus, the last display can be bounded by $2^{2k-j}$, which is summable in $j\ge 4k\ge 0$.

Here we mention that the term $\bar{M}^j_k f$ from \eqref{1401ee2.5} would also work in this case. \\

\item\textbf{Case IV: $4k\le j\le k\le 0$.} In this case we will see the main difference between $\bar{M}^j_k$ and $\widetilde{M}^j_k$. By definition,
\beq\label{08300e6.55}
\begin{split}
& m^j_k(\xi,\eta)-\widetilde{m}^j_k(\xi, \eta)=\int_1^{2}e^{i 2^k t\xi+i 2^{2k-j}t^2\eta}dt- \int_1^{2}\int_1^2 e^{i 2^k t\xi+i 2^{2k-j}\tau^2\eta}dt d\tau\\
&= \int_1^{2}\left(e^{i2^k \tau}-1\right) e^{i 2^{2k-j}\tau^2}d\tau-\int_1^{2}\int_1^2 \left(e^{i2^k t}-1\right) e^{i 2^{2k-j}\tau^2}dt d\tau.
\end{split}
\endeq
By the fundamental theorem in calculus, we bound the last display by $2^k$. Summing over $4k\le j\le k$, we obtain $|k|\cdot 2^k$, which is summable for $k\le 0$.
\end{itemize}

This finishes the proof of the $L^2$ boundedness of our operator defined in \eqref{Mmdef}.

\subsubsection{\bf $L^p$ boundedness}

In what follows we will make use of some ideas from Nagel, Stein and Wainger \cite{NSW} and Carlsson, Christ et al. \cite{CCC}. We denote
\beq
M_k^j f := \mu_k^j*f, \text{ with } \widehat{\mu_k^j} (\xi, \eta) :=m_k^j(\xi, \eta).
\endeq
A key insight in \cite{CCC}, is to compare $\mu_k^j$ with $\sigma_k^j$, where
\beq\label{0815e7.64}
\sigma_k^j:=\mu_k^j* [(\phi_k^j-\delta)\otimes (\tilde{\phi}_k^j-\delta)].
\endeq
Here $\phi$ and $\tilde{\phi}$ are two  non-negative smooth functions supported on $[-1, 1]$ having mean one, while
\beq
\phi^j_k(t):=2^{-k}\phi(2^{-k}t)\,,\:\:\:\:\:\:\:\tilde{\phi}^j_k(t):=2^{-2k-2+j}\tilde{\phi}(2^{-2k-2+j}t)\,,
\endeq
and $\delta$ is the Dirac point mass at the origin.
The meaning of the tensor product in \eqref{0815e7.64} is that its first component acts on the first variable while its second component acts on the second variable.

The difference $\mu_k^j-\sigma_k^j$ can be bounded by the strong maximal operator, by noticing that the corresponding multiplier has fast decay at infinity. In particular,
\beq\label{1302ee5.18}
\sup_{k, j}|(\mu_k^j-\sigma_k^j)*f|\lesim M_S f.
\endeq
 Thus, it only remains to bound $\sup_{k, j}|\sigma_k^j* f|$. For this, we will perform a conical Littlewood-Paley decomposition for the function $f$:
\beq
P^{\mathrm{Cone}}_k f(x, y) := \int_{\R} \hat{f}(\xi, \eta) \psi_k\Big(\frac{\xi}{\eta}\Big) e^{ix\xi+i y\eta}d\xi d\eta,
\endeq
and write
\beq\label{1002ee5.19}
\sup_{k, j}|\sigma_k^j* f|=\sup_{k, j}\Big|\sigma_k^j* \Big(\sum_{l\in \Z} P^{\mathrm{Cone}}_{k-j+l}f\Big)\Big|
\endeq
\begin{rem}
The case $l=0$ corresponds to the case where the phase function of $\widehat{\mu^j_k}=m_k^j$ has a critical point.
\end{rem}
To bound the term \eqref{1002ee5.19} on $L^p$, by the triangle inequality, it suffices to prove that
\beq\label{1002ee5.20}
\|\sup_{k, j}|\sigma_k^j* ( \pc_{k-j+l}f)|\|_{p} \lesim 2^{-\gamma_p |l|}\|f\|_p,
\endeq
for some $\gamma_p>0$. This decay in $l$ comes from the fact that away from the case $l=0$ one never sees the critical point of the phase function of $\widehat{\mu^j_k}$.

In the following, we only focus on the case $l=0$. The case of general $l\in \Z$ follows a similar approach, with the extra-twist of involving the non-stationary phase method (integration by parts).  We refer to Carlsson, Christ et al. \cite{CCC} for details. \\

Roughly speaking, estimate \eqref{1002ee5.20} follows from interpolating the $L^2$ bound of $\mc{M}_{\mathrm{dyad}}$ with some simple endpoint bound. This interpolation is again in the spirit of Nagel, Stein and Wainger \cite{NSW}. To enable this, we need to rewrite \eqref{1002ee5.20} in a slightly different way:
\beq\label{0815e7.70h}
\sup_{j, k}|\sigma_k^j* \pc_{k-j}f|=\sup_{j, k}|\sigma_k^{k-j}* \pc_{j}f|=:\sup_{j, k}|\tilde{\sigma}_k^j *\pc_{j}f|,
\endeq
where
\beq
\tilde{\sigma}_k^j:=\sigma_k^{k-j}.
\endeq
We bound \eqref{0815e7.70h} by the square function:
\beq\label{0815e7.72}
\Big( \sum_j \sup_k |\tilde{\sigma}_k^j * (\pc_{j}f)|^2\Big)^{1/2}.
\endeq
To prove the $L^p$ boundedness of the last expression, we first place it in a more general framework by considering inequalities of the form
\beq\label{0815e7.73}
\Big\| \Big( \sum_j \sup_k |\tilde{\sigma}_k^j * (\pc_{j}f)|^{q_1}\Big)^{1/q_1} \Big\|_{q_2} \lesim \|(\sum_j|\pc_j f|^{q_1})^{1/q_1}\|_{q_2}.
\endeq
For the case $q_1=\infty$ and $q_2=2$, by going back to $\sigma=\mu-(\mu-\sigma)$, it is not difficult to see that the left hand side of the last expression can be bounded by $\|\mc{M}_{\mathrm{dyad}}f\|_2+\|M_S f\|_2$, which, by the $L^2$ bound in the previous part, can be further bounded by $\|f\|_2$.

Hence, it remains to prove that for each fixed $j$ and each $q_1>1$, one has
\beq\label{0815e7.74}
\|\sup_{k}|\tilde{\sigma}_k^j * (\pc_j f)||\|_{q_1} \lesim \|\pc_j f\|_{q_1}.
\endeq
That \eqref{0815e7.72} is bounded on $L^p$ for all $p>1$ follows from an iterative interpolation argument in the spirit of \cite{NSW}.\\

Now we prove \eqref{0815e7.74}.

We bound its left hand side by a square function, and prove that
\beq\label{0815e7.75}
\|(\sum_{k}|\tilde{\sigma}_k^j *  (\pc_j f)|^2)^{1/2}|\|_{q_1} \lesim \|\pc_j f\|_{q_1}.
\endeq
By a simple anisotropic scaling, it suffices to consider the case $j=0$.

 Notice that working with the projection operator $\pc_0 f$ means we are within a frequency cone having $\{(\xi, \eta): \xi\sim \eta\}$. We continue by performing a finer frequency decomposition, as follows: denote by $\pc_{0, k}$ the frequency projection into the region $\xi\sim 2^k, \eta\sim 2^k$; then \eqref{0815e7.75} is equivalent to
\beq
\|(\sum_{k}|\tilde{\sigma}_k^0 *(\sum_{l\in \Z}\pc_{0, k+l}  f)|^2)^{1/2}|\|_{q_1} \lesim \|\pc_0 f\|_{q_1}.
\endeq
By the triangle inequality, it suffices to prove for some $\lambda_{q_1}>0$ that
\beq\label{0815e7.77}
\|(\sum_{k}|\tilde{\sigma}_k^0 *(\pc_{0, k+l}  f)|^2)^{1/2}|\|_{q_1} \lesim 2^{-\lambda_{q_1} |l|}\|\pc_0 f\|_{q_1}.
\endeq
The above estimate for $q_1=2$ follows simply from Plancherel's theorem and the mean zero property of $\tilde{\sigma}^0_k$. Thus, by interpolation and standard Littlewood-Paley theory, it suffices to prove that
\beq\label{last}
\|(\sum_{k}|\tilde{\sigma}_k^0 *(\pc_{0, k+l}  f)|^2)^{1/2}|\|_{q_1} \lesim \|(\sum_k |\pc_{0, k+l}f|^2)^{1/2} \|_{q_1}.
\endeq
As before, we first consider this estimate in a more general framework of inequalities of the form
\beq
\|(\sum_{k}|\tilde{\sigma}_k^0 *(\pc_{0, k+l}  f)|^{p_1})^{1/p_1}|\|_{q_1} \lesim \|(\sum_k |\pc_{0, k+l}f|^{p_1})^{1/p_1} \|_{q_1}.
\endeq
The case $p_1=\infty$ and $q_1=2$ follows from the $L^2$ boundedness of $\mc{M}_{\mathrm{dyad}}$ while the case $p_1=q_1>1$ is trivial. Applying the usual interpolation trick we obtain that \eqref{last} holds for all $q_1>1$. $\Box$

\appendix
\section{A local smoothing estimate}\label{AppendixA}

Let $\psi_0: \R\to \R$ be a non-negative smooth function supported on $[1/2, 3]$ with $\psi_0(t)=1$ for each $t\in [1, 2]$. Moreover, let $\varphi_0: \R\to \R$ be a non-negative smooth function supported on $[-3, 3]$ with $\varphi_0(t)=1$ for each $t\in [-1, 1]$. For a given positive real number $u$, let $A_u$ denote the averaging operator 
\beq\label{eqn:avgoperator}
A_u f(x, y) := \int_{\R} f(x-u t, y-u t^{\alpha})\psi_0(t)dt.
\endeq
Here $\alpha$ is a positive real number with $\alpha \neq 1$. For $k\in \Z$, let $P_k$ denote a Littlewood-Paley projection operator on the plane, say 
\beq
P_k f(x, y) := \int_{\R^2} e^{ix\xi+iy \eta} \hat{f}(\xi, \eta) \psi_0(\frac{1}{2^k}(\xi^2+\eta^2)^{\frac{1}{2}}) d\xi d\eta.
\endeq
Then we have 
\begin{thm}\label{0813theorem8.1}
Let $k\in \N$ be a positive integer. For each positive $\alpha\neq 1$, and each $p>2$, there exists $\gamma_{p, \alpha}>0$ such that 
\beq\label{supestimate}
\|\sup_{u\in [1, 2]} |A_u P_k f|\|_p \lesim 2^{-\gamma_{p, \alpha} \cdot k} \|f\|_p.
\endeq
Here the implicit constant depends only on $p$ and $\alpha$.
\end{thm}

In this section, we will prove Theorem \ref{0813theorem8.1} by reducing it to a decoupling inequality for cones (see Proposition \ref{weakdecoupling}) due to Bourgain \cite{Bou13} and Bourgain and Demeter \cite{BD15}. This follows the approach of Wolff \cite{Wolff00}. We will then provide a proof of the relevant decoupling inequality in the next section.

\subsection{Several reductions}

First of all, by applying the change of variable $t^{\alpha}\to s$, we see that it suffices to consider the case $\alpha>1$.  Secondly, to simplify our presentation, we will only work on the case $\alpha=2$. The other values of $\alpha>1$ can be handled in a similar way.

 We take the Fourier transform of $A_u f$:
\beq
\widehat{A_u f}(\xi, \eta)=\hat{f}(\xi, \eta) \int_{\R} e^{iut\xi+iut^{\alpha}\eta}\psi_0(t)dt.
\endeq
By the stationary phase computation (for instance see page 360 in Stein \cite{Stein} or Lemma 1.2 in Iosevich \cite{Io}), 
\beq\label{guoa.5}
\int_{\R} e^{it\xi+it^2\eta}\psi_0(t)dt=a(\xi, \eta) e^{i\frac{\xi^{2}}{\eta}}+a_{\infty}(\xi, \eta).
\endeq
Here $a_{\infty}(\xi, \eta)$ is a smooth symbol, and $a(\xi, \eta)$ is a symbol belonging to the class $S^{-\frac{1}{2}}$ and is supported on 
\beq
\{(\xi, \eta): -10\le \frac{\xi}{\eta}\le -\frac{1}{10} \}.
\endeq 
The contribution from the smooth symbol $a_{\infty}(\xi, \eta)$ can be handled via a standard argument, see for instance Stein \cite{Ste76}. We omit the details. 

Now we turn to the former term on the right hand side of \eqref{guoa.5}. Define 
\beq
T f(u, x, y) := \int_{\R^2} \hat{f}(\xi, \eta) a(u\xi, u\eta) e^{ix\xi+iy\eta+iu \frac{\xi^{2}}{\eta}} d\xi d\eta,
\endeq
and 
\beq
T_k f(u, x, y) := \psi_0(u)\cdot  T\circ P_k f(u, x, y).
\endeq
Moreover, we define a variant of $T_k$ given by
\beq\label{0823a.8}
S_k f(u, x, y) := \psi_0(u)\cdot \int_{\R^2} \widehat{P_k f}(\xi, \eta) \varphi_0(-\frac{\xi}{100\eta}) (1+\xi^2+\eta^2)^{-\frac{1}{4}} e^{ix\xi+iy\eta+iu \frac{\xi^{2}}{\eta}} d\xi d\eta.
\endeq
We will prove 
\begin{thm}\label{theorema.2}
Let $k\in \N$ be a positive integer. For each $p>2$, there exists $\gamma_p>0$ such that 
\beq\label{0823a.9}
 \|S_k f\|_{L^p(\R^3)} \lesim 2^{-(\frac{1}{p}+\gamma_p)k}\|P_k f\|_p.
\endeq
\end{thm}

We will apply Theorem \ref{theorema.2} to prove 
\beq\label{isotropicnorm}
\|\Delta_u^{\frac{s}{2}} T_k f\|_{L^p(\R^3)} \lesim 2^{-(\frac{1}{p}+\gamma_p)k+sk}\|f\|_p.
\endeq
Here 
\beq
\Delta_u^{\frac{s}{2}} F(u, x, y):= \int_{\R} e^{iu\tau} (1+|\tau|^2)^{\frac{s}{2}} \tilde{F}(\tau, x, y) d\tau,
\endeq
and $\tilde{F}$ denotes the partial Fourier transform taken in the $u$ variable only. The desired estimate \eqref{supestimate} follows from \eqref{isotropicnorm} via the fractional $L^{\infty}$ Sobolev embedding inequality 
$$\|h\|_{L_u^{\infty}(\R)}\lesim \|\Delta_u^{\frac{s}{2}}h\|_{L_u^p(\R)},$$
whenever $2\le p<\infty$ and $s>\frac{1}{p}$.

Now we come to the proof of \eqref{isotropicnorm}. By taking the partial Fourier transform of $S_k f$ in the $u$ variable, we see that  
$$(S_k f)^{\sim}(\tau, x, y)=\int_{\R^2} \widehat{P_k f}(\xi, \eta) \varphi_0(-\frac{\xi}{100\eta}) (1+\xi^2+\eta^2)^{-\frac{1}{4}} \check{\psi}_0(\tau-\frac{\xi^2}{\eta})e^{ix\xi+iy\eta} d\xi d\eta$$
is supported in the region $\{\tau\in \R: \tau\sim 2^k\}$. Hence by Young's inequality and Theorem \ref{theorema.2}, we obtain 
\beq\label{skestimate}
\|\Delta^{\frac{s}{2}}_u S_k f\|_{L^p(\R^3)} \lesim 2^{sk} \|S_k f\|_{L^p(\R^3)}\lesim 2^{-(\frac{1}{p}+\gamma_p)k+sk}\|f\|_p.
\endeq
The estimate \eqref{isotropicnorm} follows from \eqref{skestimate} via the standard H\"ormander-Mikhlin multiplier theorem, by realising that 
$$\Big| \partial_{\xi, \eta}^{\alpha} \frac{a(u\xi, u\eta)}{(1+\xi^2+\eta^2)^{-\frac{1}{4}}} \Big|\lesim_{\alpha} (1+|\xi|+|\eta|)^{-|\alpha|} \text{ for all multi-indices } \alpha\in \N^2_0,$$
uniformly in $u\in [1, 2]$.

\subsection{Proof of Theorem \ref{theorema.2} via a decoupling inequality}

Define a rescaled version of the operator $S_k$ from \eqref{0823a.8} by 
\beq
E_k f(u, x, y) := \psi_k(u)\cdot \int_{\R^2} \widehat{P_0 f}(\xi, \eta) \varphi_0(-\frac{\xi}{100 \eta}) (1+\xi^2+\eta^2)^{-\frac{1}{4}} e^{ix\xi+iy\eta+iu \frac{\xi^2}{\eta}} d\xi d\eta,
\endeq
where $\psi_k(u) :=\psi_0(\frac{u}{2^k})$. By applying a change of variables
\beq
\xi\to 2^k \xi, \eta\to 2^k \eta
\endeq
to $S_k f$ and the desired estimate \eqref{0823a.9}, we see that \eqref{0823a.9} is equivalent to 
\beq
\|E_k f\|_{L^p(\R^3)}\lesim 2^{\frac{k}{2}-\gamma_p \cdot k}\|f\|_p, \text{ for some } \gamma_p>0.
\endeq
We apply a conical frequency decomposition for $P_0 f$. Let 
\beq
\Sigma_k := \{\frac{j}{2^{\frac{k}{2}}}: j\in \Z, -100 \cdot 2^{\frac{k}{2}} \le j\le 100\cdot 2^{\frac{k}{2}}\}
\endeq
and decompose $P_0 f$ by writing
\beq
\widehat{P_0 f}=\sum_{\theta\in \Sigma_k}\widehat{f_{\theta}}:=\sum_{\theta\in \Sigma_k}  \psi_0(2^{\frac{k}{2}} (\frac{\xi}{\eta}-\theta))\cdot \widehat{P_0 f}.
\endeq

The estimate \eqref{0823a.9} in Theorem \ref{theorema.2} follows immediately from the following two results. 
\begin{prop}[Bourgain \cite{Bou13}, Bourgain and Demeter \cite{BD15}]\label{weakdecoupling}
For each $2\le p\le 4$ and each $\epsilon>0$, we have 
\beq
\|E_k f\|_{L^p(\R^3)}\lesim 2^{\frac{k}{2}(\frac{1}{2}-\frac{1}{p})+\epsilon} (\sum_{\theta\in \Sigma_k} \|E_k f_{\theta}\|_{L^p(\R^3)}^p)^{\frac{1}{p}}.
\endeq
Here the implicit constant depends only on $p$, $\alpha$ and $\epsilon$.
\end{prop}

\begin{lem}\label{organisation}
For each $p\ge 2$, we have 
\beq\label{organisation19}
(\sum_{\theta\in \Sigma_k} \|E_{k} f_{\theta}\|_{L^p(\R^3)}^p)^{\frac{1}{p}} \lesim 2^{\frac{k}{p}}\|f\|_{L^2(\R^2)}.
\endeq
\end{lem}

The proof of Proposition \ref{weakdecoupling} can be found in Bourgain \cite{Bou13} and the last section of Bourgain and Demeter \cite{BD15}. Note that here we only rely on a weak form of the decoupling inequality; that is, the exponent $p$ is in the restricted range $2\le p\le 4$, but not $4<p\le 6$. The latter is part of the main content in Bourgain and Demeter \cite{BD15}. As the proof of Proposition \ref{weakdecoupling} is short and can be made (essentially) self-contained, we will provide it in the next section.

The proof of Lemma \ref{organisation} is standard. Here we sketch the proof. By interpolation, it suffices to prove \eqref{organisation19} for $p=2$ and $p=\infty$. At $p=2$, the proof is via a simple almost orthogonality argument. For each fixed $u\in [0, 2^k]$, we have 
\beq
(\sum_{\theta\in \Sigma_k} \|E_{k} f_{\theta}(u, \cdot)\|_{L_{x, y}^2(\R^2)}^2)^{\frac{1}{2}} \lesim (\sum_{\theta\in \Sigma_k}\|f_{\theta}\|_{L^2(\R^2)}^2)^{\frac{1}{2}}\lesim \|f\|_{L^2(\R^2)}.
\endeq
In the end, we integrate in $u$ and collect the factor $2^{\frac{k}{2}}$.\\

At $p=\infty$, by Young's inequality, it suffices to show that 
\beq
\Big\|\int_{\R^2} \varphi_0(-\frac{\xi}{100 \eta})\psi_0(\eta)\psi_0(2^{\frac{k}{2}}(\frac{\xi}{\eta}-\theta)) e^{ix\xi+iy \eta+i u\frac{\xi^{2}}{\eta}}d\xi d\eta\Big\|_{L^1_{x, y}(\R^2)} \lesim 1,
\endeq
uniformly in $\theta$ and $u\in [0, 2^k]$. Without loss of generality, we take $\theta=0$. By the change of variable $2^{\frac{k}{2}}\xi\to \xi$, the last estimate is equivalent to 
\beq\label{0824a.22}
\Big\|\int_{\R^2}\psi_0(\eta) \psi_0(\frac{\xi}{\eta}) e^{ix\xi+iy \eta+i u\frac{\xi^2}{\eta}}d\xi d\eta\Big\|_{L^1_{x, y}(\R^2)} \lesim 1,
\endeq
uniformly in $u\in [0, 1]$. However, by the non-stationary phase method, when $|x|+|y|\gg 1$, we always have 
$$\Big| \int_{\R^2} \psi_0(\eta) \psi_0(\frac{\xi}{\eta}) e^{ix\xi+iy \eta+i u\frac{\xi^2}{\eta}}d\xi d\eta \Big|\lesim \frac{1}{|x|^{10}+|y|^{10}}.$$
This further implies the desired estimate \eqref{0824a.22}.

\section{The proof of a decoupling inequality}\label{AppendixB}

For a dyadic interval $\Delta\subset [0, 1]$, define the extension operator associated with $\Delta$ and the parabola $(\xi, \xi^2)$ by 
\beq
E_{\Delta} g(x) := \int_{\Delta} g(\xi) e^{ix_1 \xi+i x_2 \xi^2}d\xi.
\endeq
In this section we will prove 
\begin{thm}[Bourgain \cite{Bou13}]\label{weakparabola}
For each $2\le p\le 4$ and $\epsilon>0$, we have 
\beq\label{weakparabolapp}
\|E_{[0, 1]} g\|_{L^p(\R^2)}\lesim \delta^{-(\frac{1}{2}-\frac{1}{p}+\epsilon)}\Big(\sum_{\Delta\subset [0, 1]; l(\Delta)=\delta}\|E_{\Delta}g\|^p_{L^p(\R^2)}\Big)^{\frac{1}{p}}.
\endeq
\end{thm}

Proposition \ref{weakdecoupling} follows from Theorem \ref{weakparabola} via Fubini's theorem and an iteration argument. This iteration first appeared in the work of Pramanik and Seeger \cite{PS}. We refer to Proposition 8.1 in Bourgain and Demeter \cite{BD15} for the details. \\

In the remaining part, we will provide a proof of Theorem \ref{weakparabola}. First of all, by a simple localisation argument, and by H\"older's inequality, the estimate \eqref{weakparabolapp} follows from 
\beq\label{localisation}
\|E_{[0, 1]} g\|_{L^p(w_B)}\lesim \delta^{-\epsilon}\Big(\sum_{\Delta\subset [0, 1]; l(\Delta)=\delta}\|E_{\Delta}g\|^2_{L^p(w_B)}\Big)^{\frac{1}{2}},
\endeq
for each ball $B$ of radius $\delta^{-2}$. Here $w_B$ is a weight associated with $B$ given by 
$$w_B(x):=(1+\frac{\|x-c_B\|}{\delta^{-2}})^{-N},$$
for a large integer $N$ which will not be specified. Secondly, to prove \eqref{localisation} for all $2\le p\le 4$, by interpolation with the trivial bound at $p=2$, it suffices to look at the case $p=4$. We refer to Garrig\'os and Seeger \cite{GS10} for the details of such an interpolation argument. 

The proof of \eqref{localisation} for $p=4$ will be accomplished in three steps, which correspond to the following three subsections. 

\subsection{A bilinear restriction estimate for parabola}

The following proposition follows simply via Plancherel's theorem. 
\begin{prop}
Let $R_1, R_2\subset [0, 1]$ be two dyadic intervals with $dist(R_1, R_2)\ge \nu$ for some $\nu>0$. We have the bilinear restriction estimate
\beq\label{bilinear}
\Big\| |E_{R_1}g_1 E_{R_2}g_2|^{\frac{1}{2}} \Big\|_{L^4(\R^2)} \lesim_{\nu} \|g_1\|_2^{\frac{1}{2}}\|g_2\|_2^{\frac{1}{2}}.
\endeq 
\end{prop}
The details are left to the interested reader. 

\subsection{A bilinear decoupling inequality}

Recall that $R_1$ and $R_2$ are two dyadic intervals whose distance is not smaller than $\nu$. 
\begin{prop}\label{bilineardecoupling}
We have a bilinear version of the desired decoupling inequality \eqref{localisation}:
\beq
\Big\| (E_{R_1} g_1 E_{R_2}g_2)^{\frac{1}{2}} \Big\|_{L^4(w_B)} \lesim_{\nu} \Big( \prod_{j=1}^2 \sum_{\Delta\subset R_j: l(\Delta)=\delta} \|E_{\Delta}g_j\|_{L^4(w_B)}^2 \Big)^{\frac{1}{4}},
\endeq
for each ball $B$ of radius $\delta^{-2}$.
\end{prop}
We start by introducing some notation. Let $\tau_j$ be the $\delta^2$-neighbourhood of the parabola that lies on top of $R_j$; that is, 
$$\tau_j:=\{(\xi, \xi^2+\eta): \xi\in R_j, |\eta|\le \delta^2\}.$$
We let $\mc{P}_j$ be finitely overlapping cover of $\tau_j$ with curved regions $\theta$ of the form
$$\theta=\{(\xi, \xi^2+\eta): \xi\in [c, c+\delta], |\eta|\le \delta^2\} \text{ for some constant } c.$$
For a function $f$ supported on $\tau_j$, we let $f_{\theta}$ to denote the restriction of $f$ on $\theta$. \\

Let $f_1$ and $f_2$ be two functions supported on $\tau_1$ and $\tau_2$ separately. The bilinear estimate \eqref{bilinear} implies 
\beq
\|(\hat{f_1}\hat{f_2})^{\frac{1}{2}}\|_{L^4(w_B)} \lesim \delta^{-1} \|f_1\|_{L^2(\R^2)}^{\frac{1}{2}}\|f_2\|_{L^2(\R^2)}^{\frac{1}{2}}.
\endeq
This can be done by applying \eqref{bilinear} to each fibre $\{(\xi, \xi^2+\eta_j)\}$ for fixed $\eta_j$ ($j=1, 2$), and then applying Minkowski's inequality to the variables $\eta_j$. Now we apply the $L^2$ orthogonality and a simple localisation argument, to further obtain 
\beq
\|(\hat{f_1}\hat{f_2})^{\frac{1}{2}}\|_{L^4(w_B)} \lesim \delta^{-1} \Big(\prod_{j=1}^2 \sum_{\theta\in \mathcal{P}_j} \|\widehat{f_{j, \theta}}\|_{L^2(w_B)}^2 \Big)^{\frac{1}{4}}.
\endeq
In the end, we apply H\"older's inequality to the right hand side of the last expression
\beq
\|(\hat{f_1}\hat{f_2})^{\frac{1}{2}}\|_{L^4(w_B)} \lesim \Big(\prod_{j=1}^2 \sum_{\theta\in \mathcal{P}_j} \|\widehat{f_{j, \theta}}\|_{L^4(w_B)}^2 \Big)^{\frac{1}{4}}.
\endeq
This implies the desired estimate in Proposition \ref{bilineardecoupling} by taking $f_j=E_{R_j}g_j$.

\subsection{Bilinear decoupling implies linear decoupling}

We come to the final step of proving the desired decoupling inequality \eqref{localisation}. The idea is that the bilinear decoupling inequality in Proposition \ref{bilineardecoupling} will imply \eqref{localisation}. This is done via a simple version of the Bourgain-Guth argument from \cite{BG11}. \\

We proceed with the details. Fix a large constant $K\ll \delta^{-1}$. We split the interval $[0, 1]$ into smaller intervals of length $K^{-1}$. We use $\alpha$ to denote such an interval. Then on each ball $B_K$ of radius $K$, the function $|E_{\alpha}g|$ behaves like a constant. Hence for convenience, we will use $|E_{\alpha}g|(B_K)$ to denote such a value. 

For each given $B_K$, we let $\alpha^{*}$ denote the interval that maximises $|E_{\alpha^{*}}g|(B_K)$. We look at the collection of $\alpha$, with $dist(\alpha^{*}, \alpha)\ge \frac{1}{K}$, and 
$$|E_{\alpha}g|(B_K)\ge \frac{1}{10 K}|E_{\alpha^{*}}g(B_k)|.$$
There are two cases. The first case is that this collection is empty. Then 
\beq
|E_{[0, 1]}g|(x)\lesim |E_{\alpha^{*}}g|(B_k), \text{ for each }x\in B_k.
\endeq
The second case is that this collection contains at least one element. Call it $\alpha^{**}$. Then 
\beq
|E_{[0, 1]}g|(B_K)\lesim K^3 |E_{\alpha^{*}}g|^{\frac{1}{2}}(B_K)|E_{\alpha^{**}}g|^{\frac{1}{2}}(B_K).
\endeq
Putting these two estimates together, we obtain 
\beq\label{0824b.11}
\begin{split}
\|E_{[0, 1]}g\|_{L^4(w_{B_K})} & \le C \Big( \sum_{\alpha: l(\alpha)=\frac{1}{K}} \|E_{\alpha}g\|^2_{L^4(w_{B_K})} \Big)^{\frac{1}{2}}\\
		& + C \cdot K^3 \Big( \sum_{dist(\alpha_1, \alpha_2)\ge \frac{1}{K}} \Big\||E_{\alpha_1}g|^{\frac{1}{2}}|E_{\alpha_2}g|^{\frac{1}{2}}\Big\|^2_{L^4(w_{B_K})} \Big)^{\frac{1}{2}},
\end{split}
\endeq
for a universal constant $C$. We raise both sides of this estimate to the $4$-th power, and sum over all balls $B_K$ inside $B$, a ball of radius $\delta^{-2}$, to obtain that \eqref{0824b.11} indeed holds true with $B_K$ being replaced by $B$. Now we apply the bilinear decoupling inequality that has been proven in the previous subsection, to obtain 
\beq
\begin{split}
\|E_{[0, 1]}g\|_{L^4(w_{B})}  & \le C \Big( \sum_{\alpha: l(\alpha)=\frac{1}{K}} \|E_{\alpha}g\|^2_{L^4(w_{B})} \Big)^{\frac{1}{2}} +C_K\cdot K^{10} \Big( \sum_{\Delta: l(\Delta)=\delta} \|E_{\Delta} g\|^2_{L^4(w_B)} \Big)^{\frac{1}{2}},
\end{split}
\endeq
for a possibly larger $C$ and a constant $C_K$ depending on $K$. In the end, by invoking a parabolic rescaling, we iterate the last estimate for $\log_K(\frac{1}{\delta})$ many times, and obtain 
\beq
\|E_{[0, 1]}g\|_{L^p(w_{B})} \le C^{\log_{K}(\frac{1}{\delta})}\cdot C_K\cdot  K^{10} \Big( \sum_{\Delta: l(\Delta)=\delta} \|E_{\Delta} g\|^2_{L^p(w_B)} \Big)^{\frac{1}{2}}.
\endeq
We just need to observe that by choosing $K$ large enough, compared with $C$, we can always have the desired estimate \eqref{localisation}.


\begin{thebibliography}{CCCD99}

\bibitem[Bat13]{Bateman} Bateman, M. \emph{Single annulus $L^p$ estimates for Hilbert transforms along vector fields.}  Rev. Mat. Iberoam. 29 (2013), no. 3, 1021-1069.

\bibitem[BT13]{BT} Bateman, M. and Thiele, C. \emph{$L^p$ estimates for the Hilbert transforms along a one-variable vector field.}  Anal. PDE 6 (2013), no. 7, 1577-1600.

\bibitem[Bou86]{Bou86} Bourgain, J. \emph{Averages in the plane over convex curves and maximal operators.} J. Analyse Math. 47 (1986), 69-85. 

\bibitem[Bou89]{Bo} Bourgain, J. \emph{A remark on the maximal function associated to an analytic vector field.} Analysis at Urbana, Vol. I (Urbana, IL, 1986-1987), 111-132, London Math. Soc. Lecture Note Ser., 137, Cambridge Univ. Press, Cambridge, 1989.

\bibitem[Bou13]{Bou13} Bourgain, J. \emph{Moment inequalities for trigonometric polynomials with spectrum in curved hypersurfaces.} Israel J. Math. 193 (2013), no. 1, 441-458.

\bibitem[BD15]{BD15} Bourgain, J. and Demeter, C. \emph{The proof of the $l^2$ decoupling conjecture.} Ann. of Math. (2) 182 (2015), no. 1, 351-389. 

\bibitem[BG11]{BG11} Bourgain, J. and  Guth, L. {\em Bounds on oscillatory integral operators based on multilinear estimates}. Geom. Funct. Anal. 21 (2011), no. 6, 1239-1295

\bibitem[CSWW99]{CSWW} Carbery, A., Seeger, A., Wainger, S. and Wright, J. \emph{Classes of singular integral operators along variable lines.} J. Geom. Anal. 9 (1999), no. 4, 583-605. 

\bibitem[CCCD99]{CCC} Carlsson, H., Christ, M., Cordoba, A., Duoandikoetxea, J., Rubio de Francia, J. L., Vance, J., Wainger, S. and Weinberg, D. \emph{$L^p$ estimates for maximal functions and Hilbert transforms along flat convex curves in $R^2$.} Bull. Amer. Math. Soc. (N.S.) 14 (1986), no. 2, 263-267. 
 
\bibitem[CNSW99]{CNSW} Christ, M., Nagel, A., Stein, E. and Wainger, S. \emph{Singular and maximal Radon transforms: analysis and geometry.} Ann. of Math. (2) 150 (1999), no. 2, 489-577. 



\bibitem[FS71]{FeffermanStein} Fefferman, C and Stein, E. \emph{Some Maximal Inequalities.}  American Journal of Mathematics. Vol. 93, No. 1 (Jan., 1971), pp. 107-115.

\bibitem[GS10]{GS10} Garrig\'os, G. and Seeger, A. \emph{A mixed norm variant of Wolff's inequality for paraboloids.} Harmonic analysis and partial differential equations, 179-197, Contemp. Math., 505, Amer. Math. Soc., Providence, RI, 2010. 



\bibitem[Guo16]{Guo2} Guo, S. \emph{Single annulus estimates for the variation-norm Hilbert transforms along Lipschitz vector fields.} To appear in the Proc. Amer. Math. Soc.

\bibitem[GPRY16]{GPRY} Guo, S., Pierce, L., Roos, J. and Yung, P.-L. \emph{Polynomial Carleson operators along monomial curves in the plane.} arXiv:1605.05812






\bibitem[HKY09]{HKY} Hong, S., Kim, J. and Yang, C. \emph{Maximal operators associated with vector polynomials of lacunary coefficients.} Canad. J. Math. 61 (2009), no. 4, 807-827.

\bibitem[Ios94]{Io} Iosevich, A. \emph{Maximal operators associated to families of flat curves in the plane.} Duke Math. J. 76 (1994), no. 2, 633-644.



\bibitem[Kar07]{Kara} Karagulyan, G. \emph{On unboundedness of maximal operators for directional Hilbert transforms.} Proc. Amer. Math. Soc. 135 (2007), no. 10, 3133-3141 (electronic).

\bibitem[LL06]{LL1} Lacey, M. and Li, X. \emph{Maximal theorems for the directional Hilbert transform on the plane.} Trans. Amer. Math. Soc. 358 (2006), no. 9, 4099-4117.

\bibitem[LL10]{LL2} Lacey, M. and Li, X. \emph{On a conjecture of E. M. Stein on the Hilbert transform on vector fields.}  Mem. Amer. Math. Soc. 205 (2010), no. 965, viii+72 pp.


\bibitem[Lie09]{Lie1} Lie, V. \emph{The (weak-$L^2$) boundedness of the quadratic Carleson operator.} Geom. Funct. Anal. 19 (2009), no. 2, 457-497.

\bibitem[Lie11]{Lie2} Lie, V. \emph{The Polynomial Carleson Operator.} arXiv:1105.4504.

\bibitem[Lie15]{Lie3} Lie, V. \emph{The boundedness of the Bilinear Hilbert Transform along non-flat smooth curves.} Amer. J. Math., 137(2), pp. 313-364, 2015.

\bibitem[Lie16]{Lie4} Lie, V. \emph{On the Boundedness of the Bilinear Hilbert Transform along ��non-flat�� smooth curves. The Banach Triangle case ($L^r$, $1\leq r<\infty$).} Rev. Mat. Iberoam, to appear 2016.

\bibitem[MR98]{MR} Marletta, G. and Ricci, F. \emph{Two-parameter maximal functions associated with homogeneous surfaces in $R^n$.} Studia Math. 130 (1998), no. 1, 53-65.

\bibitem[MSS92]{MSS92} Mockenhaupt, G., Seeger, A. and Sogge, C. \emph{Wave front sets, local smoothing and Bourgain's circular maximal theorem.} Ann. of Math. (2) 136 (1992), no. 1, 207-218. 

\bibitem[Mus14]{Mu} Muscalu, C. \emph{Calderon commutators and the Cauchy integral on Lipschitz curves revisited I. First commutator and generalizations.} Rev. Mat. Iberoam. 30 (2014), no. 2, 727-750.

\bibitem[NSW78]{NSW} Nagel, A., Stein, E. and Wainger, S. \emph{Differentiation in lacunary directions.} Proc. Nat. Acad. Sci. U.S.A. 75 (1978), no. 3, 1060-1062.

\bibitem[NSW79]{NSW78} Nagel, A., Stein, E. and Wainger, S. \emph{Hilbert transforms and maximal functions related to variable curves.} Harmonic analysis in Euclidean spaces (Proc. Sympos. Pure Math., Williams Coll., Williamstown, Mass., 1978), Part 1, pp. 95-98, Proc. Sympos. Pure Math., XXXV, Part, Amer. Math. Soc., Providence, R.I., 1979. 

\bibitem[OSTTW12]{OSTTW12} Oberlin, R., Seeger, A., Tao, T., Thiele, C. and Wright, J. \emph{A variation norm Carleson theorem.} J. Eur. Math. Soc. (JEMS) 14 (2012), no. 2, 421-464. 

\bibitem[PS07]{PS} Pramanik, M. and Seeger, A. \emph{$L^p$ regularity of averages over curves and bounds for associated maximal operators.} Amer. J. Math. 129 (2007), no. 1, 61-103. 



\bibitem[See94]{See94} Seeger, A. \emph{$L^2$-estimates for a class of singular oscillatory integrals.} Math. Res. Lett. 1 (1994), no. 1, 65-73. 

\bibitem[Ste76]{Ste76} Stein, E. \emph{Maximal functions. I. Spherical means.} Proc. Nat. Acad. Sci. U.S.A. 73 (1976), no. 7, 2174-2175. 

\bibitem[Ste93]{Stein} Stein, E. \emph{Harmonic analysis: real-variable methods, orthogonality, and oscillatory integrals. With the assistance of Timothy S. Murphy.} Princeton Mathematical Series, 43. Monographs in Harmonic Analysis, III. Princeton University Press, Princeton, NJ, 1993. xiv+695 pp.

\bibitem[SS12]{SS} Stein, E. and Street, B. \emph{Multi-parameter singular Radon transforms III: Real analytic surfaces.} Adv. Math. 229 (2012), no. 4, 2210-2238.

\bibitem[SW01]{SW} Stein, E. and Wainger, S. \emph{Oscillatory integrals related to Carleson's theorem.} Math. Res. Lett. 8 (2001), no. 5-6, 789-800.

\bibitem[SW03]{SW03} Seeger, A. and Wainger, S. \emph{Singular Radon transforms and maximal functions under convexity assumptions.} Rev. Mat. Iberoamericana 19 (2003), no. 3, 1019-1044.



\bibitem[Wol00]{Wolff00} Wolff, T. \emph{Local smoothing type estimates on $L^p$ for large $p$. } Geom. Funct. Anal. 10 (2000), no. 5, 1237-1288.

\end{thebibliography}
\end{document}